\documentclass[10pt,a4paper]{amsart}
\usepackage[utf8]{inputenc}
\usepackage{amsmath}
\usepackage{amsfonts}
\usepackage{amssymb}
\usepackage{amsthm}
\usepackage{amscd}
\usepackage{float}
\usepackage{tikz}
\usepackage{graphicx}
\usepackage[colorlinks=true]{hyperref}
\hypersetup{urlcolor=blue, citecolor=red}
\usepackage{hyperref}

\newtheorem{theorem}{Theorem}[section]
\newtheorem{corollary}[theorem]{Corollary}

\newtheorem{lemma}[theorem]{Lemma}
\newtheorem{proposition}[theorem]{Proposition}

\theoremstyle{definition}
\newtheorem{definition}[theorem]{Definition}
\newtheorem{remark}[theorem]{Remark}

\newcommand{\R}{\mathbb{R}}
\newcommand{\N}{\mathbb{N}}
\newcommand{\C}{\mathbb{C}}
\newcommand{\Z}{\mathbb{Z}}
\newcommand{\K}{\mathbb{K}}

\def\Xint#1{\mathchoice
{\XXint\displaystyle\textstyle{#1}}%
{\XXint\textstyle\scriptstyle{#1}}%
{\XXint\scriptstyle\scriptscriptstyle{#1}}%
{\XXint\scriptscriptstyle\scriptscriptstyle{#1}}%
\!\int}
\def\XXint#1#2#3{{\setbox0=\hbox{$#1{#2#3}{\int}$ }
\vcenter{\hbox{$#2#3$ }}\kern-.6\wd0}}

\def\dashint{\Xint-}

\makeatletter
\@namedef{subjclassname@2020}{%
  \textup{2020} Mathematics Subject Classification}
\makeatother

\begin{document}

\title[Decoupling for complex curves]{Decoupling for complex curves and improved decoupling for the cubic moment curve}
\author{Robert Schippa}
\email{rschippa@berkeley.edu}

\address{UC Berkeley, Department of Mathematics, 847 Evans Hall
Berkeley, CA 94720-3840}

\begin{abstract} We prove sharp $\ell^2$-decoupling inequalities for non-degenerate complex curves via the bilinear argument due to Guo--Li--Yung--Zorin-Kranich, which in turn is inspired by the efficient congruencing argument of Wooley.
Secondly, quantifying the iteration in the cubic case, we obtain a logarithmic refinement of the decoupling inequality for the cubic moment curve.
\end{abstract}

\maketitle

\section{Introduction}

%
%
%

\subsection{Decoupling inequalities for non-degenerate complex curves}

In the following we prove decoupling inequalities for non-degenerate complex curves. An example is the complex moment curve:
\begin{equation*}
\C \ni z \mapsto \gamma(z) = (z,\frac{z^2}{2},\frac{z^3}{3!},\ldots,\frac{z^k}{k!}) \in \C^k \equiv \R^{2k}.
\end{equation*}
By non-degeneracy we mean that the ``complex torsion"
\begin{equation*}
|\partial_z \gamma \wedge_\C \partial^2_z \gamma \ldots \wedge_\C \partial_z^k \gamma | \gtrsim 1
\end{equation*}
is uniformly bounded from below. 
$\gamma(z) \in \C^k$ is translated to an element of $\R^{2k}$ by separating real and imaginary part of the complex components. We define $F: \C^k \to \R^{2k}$ via
\begin{equation}
\label{eq:CorrespondenceMappingRC}
F(z_1,\ldots,z_k) = (\Re z_1, \Im z_1, \Re z_2, \Im z_2, \ldots, \Re z_k , \Im z_k).
\end{equation}

 Write $z = s + it$, $s,t \in \R$ with $i^2 = -1$, and we find
\begin{equation*}
F(z,\frac{z^2}{2},\frac{z^3}{3!},\ldots) = (s,t,\frac{s^2-t^2}{2},st,\frac{s^3 - 3st^2}{3!}, \frac{3s^2 t- t^3}{3!}, \ldots ).
\end{equation*}
We shall formulate the results entirely in Euclidean space, but to take advantage of the complex structure carry out several computations using the complex notations.

\medskip

Recall the canonical (anisotropic) neighbourhood for the real moment curve $\gamma_{\R}(s) = (s,\frac{s^2}{2},\ldots,\frac{s^k}{k!}) \in \R^k$. For an interval $I_\delta \subseteq [0,1]$ of length $2 \delta$ with center $c_{I_\delta}$ we associate the anisotropic neighbourhood:
\begin{equation}
\label{eq:RealAnisotropicNeighbourhood}
\theta_{I_\delta, \R} = \{ \gamma_{\R}(c_{I_\delta}) + a_1 \partial_s \gamma_{\R}(c_{I_\delta}) + \ldots + a_n \partial_s^n \gamma_{\R}(c_{I_\delta}) \; : |a_i| \leq 4 \delta^{i} \text{ for } i=1,\ldots,n \}.
\end{equation}
A function $f \in \mathcal{S}(\R^k)$ with $\text{supp} (\hat{f}) \subseteq \theta_{I_\delta,\R}$ is essentially constant on dual anisotropic rectangles of size $\delta^{-1} \times \ldots \times \delta^{-k}$. 
Let $g:[0,1] \to \C$ be supported in an interval of length $\delta$.
Fourier extension applied to $g$
\begin{equation*}
\mathcal{E}_{\gamma_{\R}} g(x) = \int_{[0,1]} e^{i x \cdot \gamma_{\R}(s)} g(s) ds
\end{equation*}
has on balls of size $\delta^{-n}$ the properties of $f$. Let $\mathcal{I}_{\delta}$ denote an essentially disjoint covering of $[0,1]$ with $\delta$ intervals. For $I \in \mathcal{I}_{\delta}$ with center $c_{I}$ we assign the anisotropic neighbourhood $\theta_{I,\R}$ from above. By $f_{\theta}$ we denote the Fourier projection of $f$ to $\theta$.

\smallskip

In \cite{BourgainDemeterGuth2016} Bourgain--Demeter--Guth proved the following decoupling inequality for curves with torsion\footnote{Strictly speaking, the decoupling inequality was only proved for the real moment curve, but the extension is immediate by reiteration, i.e., the Pramanik--Seeger \cite{PramanikSeeger2007} argument.}:
\begin{theorem}[$\ell^2$-decoupling for curves with torsion]
\label{thm:RealDecouplingMomentCurve}
Let $k \geq 1$ and $\gamma_{\R}:[0,1] \to \R^k$ denote the real moment curve. 
Let $f \in \mathcal{S}(\R^k)$ with $\text{supp}(\hat{f}) \subseteq \bigcup_{I \in \mathcal{I}_{\delta}} \theta_{I,\R}$. Then the following estimate holds:
\begin{equation}
\label{eq:RealMomentCurveDecoupling}
\| f \|_{L^{p_k}(\R^k)} \lesssim_\varepsilon \delta^{-\varepsilon} \big( \sum_{I \in \mathcal{I}_{\delta}} \| f_{{\theta}_{I,\R}} \|_{L^{p_k}(\R^k)}^2 \big)^{\frac{1}{2}}.
\end{equation}
\end{theorem}

\medskip

In this note we prove the extension of the above theorem to non-degenerate complex curves. To formulate the result, we need to consider the anisotropic neighbourhoods of a complex non-degenerate curve. Let $Q_1 = \{ z \in \C : 0 \leq \Re z, \Im z \leq 1 \}$.

We make the following definition:
\begin{definition}
A complex curve $\gamma : \C \to \C^n$ is called non-degenerate (or complex curve with torsion) if it is analytic in $Q_1$ and if there is $c_{\gamma} > 0$ such that
\begin{equation*}
|\partial_z \gamma(z) \wedge \partial^2_z \gamma(z) \wedge \ldots \wedge \partial_z^n \gamma(z) | \geq c_\gamma > 0 \; \forall z \in Q_1.
\end{equation*}
\end{definition}
For future reference we define the constant:
\begin{equation*}
C_\gamma = \sup_{1 \leq m \leq 2(k+1)} \| \partial^m_z \gamma \|_{L^\infty(Q_1)}.
\end{equation*}
By a Pramanik--Seeger argument we can always normalize the complex curves, making them ``close" in a quantitative sense to the complex moment curve.


\medskip

The canonical (anisotropic) neighbourhood of the image of $((s,t) \mapsto F(\gamma(s+it)) \in \R^{2k})$ is determined by the almost constant property of $\mathcal{E}_\gamma f(\underline{x})$, $f \in L^1([0,1]^2)$, where
\begin{equation*}
\mathcal{E}_\gamma f(\underline{x}) = \int_{[0,1]^2} e^{i \underline{x} \cdot F(\gamma(s+it))} f(s,t) ds dt.
\end{equation*}

For a square $Q \subseteq \R^2$ we denote the side length by $\ell(Q)$.
 Let $Q_\delta \subseteq [0,1]^2$ be a square in $[0,1]^2$ of length $\delta$. We carry out a (complex) Taylor expansion of $\gamma(z)$, $z \in Q_\delta$, around the center of $Q_\delta$ given by $c_z = s + it$:
\begin{equation*}
\gamma(z) = \gamma(c_z) + \partial_z \gamma(c_z) \cdot (z-c_z) + \partial^2_{z} \gamma(c_z) \frac{(z-c_z)^2}{2} + \ldots.
\end{equation*}
Separate $(z-c_z)^k = z^{(k)}_{\Re} + i z^{(k)}_{\Im}$. 
For $1 \leq m \leq k$ we define the ``realized" tangent vectors by
\begin{equation}
\label{eq:RealizedTangentVectors}
F(\partial^m_z \gamma(c_z)) = \Gamma^m_1(c_z,\gamma), \quad F(i \partial^m_z \gamma(c_z)) = \Gamma^m_2(c_z,\gamma).
\end{equation}

Since $|z^{(m)}_{\Re}|$, $|z^{(m)}_{\Im}| \lesssim \delta^m$, we find that $|\mathcal{E}_\gamma f| $ is essentially constant on translates of dimensions $\sim \delta^{-1} \times \delta^{-1} \times \delta^{-2} \times \delta^{-2} \times \ldots$ in directions $\bigotimes_{j=1}^k \Gamma^j_1(c_z,\gamma) \times \Gamma^j_2(c_z,\gamma) $.

\smallskip

Let $\mathcal{Q}_\delta = (Q_{J})_{J \in \mathcal{J}_\delta}$ be an essentially disjoint partition of $[0,1]^2$ into squares with side length $\delta$ and with centers aligned on a grid. For $c_{Q_J}$ denoting the center of $Q_{J}$ let $\theta_{Q_{J},\gamma}$ denote the analog of the real anisotropic neighbourhood of $\theta_{I_\delta,\R}$ from \eqref{eq:RealAnisotropicNeighbourhood} dictated by $\Gamma^i_j(c_{Q_J},\gamma)$; we refer to Section \ref{section:Preliminaries} for the definition.
The collection of $\theta_{Q_{J},\gamma}$ with $Q_J \in \mathcal{Q}_\delta$ is denoted by $\Theta_{\delta,\gamma}$. 

\begin{theorem}[$\ell^2$-decoupling for non-degenerate complex curves]
\label{thm:DecouplingComplexCurves}
Let $k \in \N$, $\gamma:\C \to \C^k$ be a non-degenerate complex curve and $f \in \mathcal{S}(\R^{2k})$ with $\text{supp}(\hat{f}) \subseteq \bigcup_{Q \in \mathcal{Q}_\delta} \theta_{Q,\gamma}$. The following estimate holds:
\begin{equation}
\label{eq:ComplexDecouplingIntroduction}
\| f \|_{L^p(\R^{2k})} \lesssim_{\varepsilon,c_\gamma,C_\gamma} \delta^{-\varepsilon} \big( \sum_{Q \in \mathcal{Q}_\delta} \| f_{\theta_{Q}} \|_{L^p(\R^{2k})}^2 \big)^{\frac{1}{2}}
\end{equation}
with $p_k = k(k+1)$.
\end{theorem}

For $k=1$ \eqref{eq:RealMomentCurveDecoupling} and \eqref{eq:ComplexDecouplingIntroduction}  is immediate from Plancherel's theorem. For $k=2$, \eqref{eq:RealMomentCurveDecoupling} is a special case of the decoupling inequalities for elliptic surfaces proved by Bourgain--Demeter \cite{BourgainDemeter2015}. As a consequence of \eqref{eq:RealMomentCurveDecoupling}, we obtain Vinogradov's Mean-Value Theorem on the numbers of solutions to a system of diophantine inequalities. This was proved in the cubic case already by Wooley \cite{WooleyA,WooleyB} via efficient congruencing.

\medskip

Combining the arguments of Bourgain--Demeter--Guth and Wooley \cite{Wooley2019}, Guo--Li--Yung--Zorin-Kranich \cite{GuoLiYungZorinKranich2021} found a concise proof of the $\ell^2$-decoupling inequality for the moment curve at the endpoint $p_k = k(k+1)$ in all dimensions. This bilinear decoupling is based on induction on dimension and transversality and will be used presently in the complex case.

By standard means, the sharp decoupling inequality
gives bounds on solutions to the system of diophantine equations associated with the complex moment curve:
\begin{equation}
\label{eq:DiophantineSystem}
\left\{ \begin{array}{cl}
k_1 + \ldots + k_s &= k_{s+1} + \ldots + k_{2s}, \\
j_1 + \ldots + j_s &= j_{s+1} + \ldots + j_{2s}, \\
k_1^2 - j_1^2 + \ldots + k_s^2 - j_s^2 &= k_{s+1}^2 - j_{s+1}^2 + \ldots + k_{2s}^2 - j_{2s}^2, \\
k_1 j_1 + \ldots + k_s j_s &= k_{s+1} j_{s+1} + \ldots + k_{2s} j_{2s}, \\
\sum_{m=1}^s (F(s_m+i t_m))_{l} &= \sum_{m=s+1}^{2s} (F(s_m + i t_m))_l \text{ for } l=5,\ldots,2s.
\end{array} \right.
\end{equation}

It must be pointed out that Wooley \cite[Corollary~3.2]{Wooley2019} obtained $p$-adic decoupling estimates via efficient congruencing, which by restriction of scalars \cite[Section~16]{Wooley2019} yields the sharp bounds of solutions (i.e., the concentration along the diagonal) of the above system \eqref{eq:DiophantineSystem}.

The decoupling inequality \eqref{eq:ComplexDecouplingIntroduction} recovers a consequence of \cite[Section~16]{Wooley2019}:
\begin{corollary}
Let $k \geq 1$ and $p_k = k(k+1)$, $2s = p_k$. Then the system of equations \eqref{eq:DiophantineSystem} with integers $1 \leq k_1,\ldots,k_{2s} \leq N$ and $1 \leq j_1,\ldots,j_{2s} \leq N$ has $\lesssim_\varepsilon N^{2s+\varepsilon}$ solutions.
\end{corollary}

\begin{remark}
This insight was not reflected in the previous preprint version \\ arXiv:2302.10884v2, where Theorem~1.4 was erroneously presented as novel result. Trevor Wooley kindly shared the above elaboration with the author.
\end{remark}

We note that the components of $F(\gamma(z))$ are in general not elliptic: We have e.g. the indefinite components
\begin{equation*}
F(\gamma(s+it))_3= s^2 - t^2, \quad F(\gamma(s+it))_4 = st.
\end{equation*}
Nonetheless, the $\ell^2 L^6$-decoupling holds true, and we have the essentially sharp bound for diophantine solutions to the system of equations:
\begin{equation*}
\left\{ \begin{array}{cl}
k_1 + k_2 + k_3 &= k_4 + k_5 + k_6, \\
j_1 + j_2 + j_3 &= j_4 + j_5 + j_6, \\
k_1^2 - j_1^2 + k_2^2 - j_2^2 + k_3^2 - j_3^2 &= k_4^2 - j_4^2 + k_5^2 - j_5^2 + k_6^2 - j_6^2, \\
k_1 j_1 + k_2 j_2 + k_3 j_3 &= k_4 j_4 + k_5 j_5 + k_6 j_6.
\end{array} \right.
\end{equation*}
Recall that sharply divergent from the present results, for the hyperbolic paraboloid $f(s,t) = s^2 - t^2$ the $\ell^2 L^4$-estimate for the elliptic paraboloid fails dramatically \cite{BourgainDemeter2017} since the surface contains a straight line. In the present situation the complex torsion reflects a transversality, which salvages the estimates from the elliptic case.

\begin{remark}
In \cite{Schippa2025} a sharp square function estimate ($L^4 \ell^2$) was obtained for the complex paraboloid and its conical extension. This square function estimate yields sharp smoothing estimates for averages over curves. We hope to extend the presently proved decoupling estimates
and consider applications to smoothing estimates in future work.
\end{remark}

\subsection{Logarithmic refinement in the cubic case}

We take the opportunity to quantify the bilinear decoupling iteration due to Guo--Li--Yung--Zorin-Kranich \cite{GuoLiYungZorinKranich2021} in the cubic case, which yields a logarithmic refinement of Vinogradov's Mean-Value Theorem. In the following we denote by $\mathcal{D}_{k,\R}(\delta)$ the decoupling constant for the real moment curve in $k$ (real) dimensions.
\begin{theorem}[Improved~decoupling~for~the~moment~curve~in~three~dimensions]
\label{thm:ImprovedDecoupling}
There is $0<\delta_0<1$ and $C > 0$ such that for $0<\delta<\delta_0$, we have the following bound for the decoupling constant of the real moment curve in three dimensions:
\begin{equation}
\label{eq:DecouplingConstant3d}
\mathcal{D}_{3,\R}(\delta) \leq \exp ( C \frac{\log \delta^{-1}}{ \log(\log(\delta^{-1}))}).
\end{equation}
\end{theorem}

Bourgain--Demeter \cite{BourgainDemeter2015} proved that for any $\varepsilon >0$ there is $C_\varepsilon >0$ such that
\begin{equation*}
\mathcal{D}_{2,\R}(\delta) \leq C_\varepsilon \delta^{-\varepsilon}.
\end{equation*}

Li \cite[Theorem~1.1]{Li2021} (see also \cite{Li2020}) observed how the double exponential bound
\begin{equation}
\label{eq:DoubleExponential}
\mathcal{D}_{2,\R}(\delta) \leq A^{A^{\frac{1}{\varepsilon}}} \delta^{-\varepsilon}
\end{equation}
allows one to sharpen the decoupling constant to
\begin{equation*}
\mathcal{D}_{2,\R}(\delta) \leq \exp( C \frac{\log \delta^{-1}}{\log(\log \delta^{-1})}).
\end{equation*}

Li \cite{Li2021} proved \eqref{eq:DoubleExponential} via a bilinear approach motivated by efficient congruencing. More recently, Guth--Maldague--Wang \cite{GuthMaldagueWang2024} improved \eqref{eq:DoubleExponential} to
\begin{equation}
\label{eq:LogarithmicDecoupling}
\mathcal{D}_{2,\R}(\delta) \leq \log(\delta^{-1})^c
\end{equation}
for some (possibly large) constant $c$. It is conjectured that $c=\frac{1}{6}$ is sharp due to a Gauss sum example given by Bourgain \cite{Bourgain1993I}. The approach in \cite{GuthMaldagueWang2024} differs from \cite{Li2021} as it uses a high-low decomposition. It remains an open question to use the high-low approach in higher dimensions. 
\medskip

\emph{Outline of the paper.} In Section \ref{section:Preliminaries} we introduce more notations and reduce to normalized curves. In Section \ref{section:ProofComplexCurves} we show Theorem \ref{thm:DecouplingComplexCurves} by extending the argument from \cite{GuoLiYungZorinKranich2021}.
The key ingredient are lower-dimensional decoupling inequalities derived from transversality. Here we need to translate the complex notions of transversality to Euclidean geometry.
Finally, in Section \ref{section:LogarithmicImprovement} we show the logarithmic improvement of the decoupling inequality for the cubic moment curve.

\section{Preliminaries}
\label{section:Preliminaries}

\subsection{Real and complex notations}
\label{subsection:RealComplexNotations}
Let $k \in \N$, and $F: \C^k \to \R^{2k}$ denote the $\R$-linear mapping defined in \eqref{eq:CorrespondenceMappingRC}.
By $\wedge_\C$ we denote the complex ($\C$-linear) wedge product in the exterior algebra
\begin{equation*}
 \Lambda(\C^k) = \big\langle e_{i_1} \wedge_{\C} \ldots \wedge_{\C} e_{i_m} \, : \, 1 \leq m \leq k, \; 1 \leq i_1 < i_2 < \ldots < i_m \leq k \big\rangle_{\C}.
\end{equation*}
By $\wedge_\C: \Lambda(\C^k) \times \Lambda(\C^k) \to \Lambda(\C^k)$ we denote the wedge product in the exterior algebra over $\C^k$. By $\wedge_\R: \Lambda(\R^{2k}) \times \Lambda(\R^{2k}) \to \Lambda(\R^{2k})$ we denote the $\R$-linear wedge product. Recall the ``realized" higher-order tangent vectors defined by
\begin{equation*}
F(\partial_z^m \gamma(z)) = \Gamma_1^m(z,\gamma), \quad F(i \partial_z^m \gamma(z)) = \Gamma^m_2(z,\gamma) \text{ for } m=1,\ldots, k.
\end{equation*}

Let $(Q_{J})_{J} = \mathcal{Q}_\delta$ be an essentially disjoint partition of $[0,1]^2$ into squares with side length $\delta$ and with centers aligned on a grid. For a non-degenerate curve $\gamma$ and $c_{J}$ denoting the center of $Q_{J}$ we let
\begin{equation*}
\begin{split}
\theta_{Q_{J},\gamma} &= \{ \gamma(c_{J}) + a_{11} \Gamma^1_1(c_{J},\gamma)  + a_{12} \Gamma^1_2(c_{J},\gamma)  + a_{21} \Gamma^2_1(c_{J},\gamma) \\
 &\quad + a_{22} \Gamma^2_2(c_{J},\gamma) + \ldots : 
 \, |a_{ij}| \leq 4 \delta^i \text{ for } i=1,\ldots,k \text{ and } j =1,2 \}.
 \end{split}
\end{equation*}
The collection of $\theta_{Q_{J},\gamma}$ for $Q_J \in \mathcal{Q}_\delta$ is denoted by $\Theta(\delta,\gamma)$.
Let $\mathcal{N}_\delta(\gamma) = \bigcup_{J \in \mathcal{J}_\delta} \theta_{Q_{J,\gamma}}$ denote the anisotropic neighbourhood at scale $\delta$. We shall consider decoupling inequalities for $f \in \mathcal{S}(\R^{2k})$ with $\text{supp} (\hat{f}) \subseteq \mathcal{N}_\delta(\gamma)$. By $f_\theta$ we denote the Fourier projection of $f$ to $\theta \in \Theta(\delta,\gamma)$. For $\bar{Q} \subseteq [0,1]^2$ we let
\begin{equation*}
\Theta(\delta,\bar{Q}) = \{ \theta_{Q} : Q \in \mathcal{Q}_\delta, \; Q \subseteq \bar{Q} \}.
\end{equation*}


\subsection{Normalized non-degenerate curves}
\label{subsection:NormalizedCurves}
The induction will be carried out over the following family of curves:
\begin{definition}[Normalized~non-degenerate~curve]
Let $\gamma: \C \to \C^k$ be an analytic curve in $Q_1$. $\gamma$ is referred to as non-degenerate normalized curve if
\begin{equation*}
\frac{1}{4} \leq |\partial_z \gamma(z) \wedge_{\C} \ldots \wedge_{\C} \partial_z^k \gamma(z) | \leq 4,
\end{equation*}
and
\begin{equation*}
|\partial^j_z \gamma(z)| \leq 2 k \text{ for } 1 \leq j \leq k \text{ and } |\partial_z^{k+1} \gamma(z) | \leq \frac{2^{-10k}}{(2k)^k k!}.
\end{equation*}
\end{definition}

Let $\gamma: \C \to \C^k$ be a non-degenerate complex curve. By finitely many Taylor expansions in $Q_1$ and affine transformations, we can reduce the decoupling to normalized non-degenerate curves. Indeed, an application of Taylor's theorem in a cube centered at $z_0$ with side length $\delta$ yields
\begin{equation*}
\gamma(z) = \gamma(z_0) + \sum_{m=1}^k \partial_z^m \gamma(z_0) \frac{(z-z_0)^m}{m!} +\partial_z^{k+1} \gamma(z_1) \frac{(z-z_0)^{k+1}}{(k+1)!}.
\end{equation*}
We let $\delta z' = z-z_0$ and carry out a linear change of variables parametrized by $A \in \R^{2k,2k}$ which is defined by:
\begin{equation*}
A( \delta^m F(\partial_z^m \gamma(z_0))) = e_{2m-1}, \quad A(\delta^m F(i \partial_z^m \gamma(z_0))) = e_{2m} \text{ for } m=1,\ldots,k.
\end{equation*}
Letting $A^t \underline{x}' = \underline{x}$, we find
\begin{equation*}
\underline{x} \cdot F(\gamma(z)) = \underline{x}' \cdot F((z',(z')^2/2!,\ldots,(z')^k/k!) + p(z')) + \text{Lin}(x')
\end{equation*}
with 
\begin{equation*}
\sup_{z' \in Q_1} | \partial_{z'}^j p(z') | \leq \delta C_\gamma \text{ for } j=1,\ldots,k+1.
\end{equation*}
Let $\bar{\gamma}(z') = (z',(z')^2/2,\ldots,(z')^k/k!) + p(z')$. Choosing $\delta = \delta(C_\gamma)$ we can satisfy the estimates for the derivatives. This incurs the dependence on $C_\gamma$ in \eqref{eq:ComplexDecouplingIntroduction}. 
Secondly, the change of variables by $A$ is non-singular by non-degeneracy of $\gamma$, which incurs the dependence on $c_\gamma$ in \eqref{eq:ComplexDecouplingIntroduction}.

\subsection{Decoupling constants}

After reducing to normalized non-degenerate curves defined in Subsection \ref{subsection:NormalizedCurves}, we consider decoupling constants for this class of curves.

\smallskip

We define the critical exponent (in $k$ complex, i.e. $2k$ real dimensions) by
\begin{equation*}
p_k = k(k+1),
\end{equation*}
and the linear decoupling constant at scale $\delta$ denoted by $\mathcal{D}_k(\delta)$ is defined as
\begin{equation*}
\begin{split}
\mathcal{D}_k(\delta) &= \inf \{ C > 0 : \| f \|_{L^{p_k}(\R^{2k})} \leq C \big( \sum_{\theta \in \Theta(\delta,\gamma)} \| f_\theta \|^2_{L^{p_k}(\R^{2k})} \big)^{\frac{1}{2}}, \quad \text{supp}(\hat{f}) \subseteq \mathcal{N}_{\delta,\gamma} \\
&\quad \quad \quad \text{ for a normalized non-degenerate curve } \gamma \}.
\end{split}
\end{equation*}

We note the following consequence of anisotropic rescaling:
\begin{lemma}[Anisotropic~Rescaling]
\label{lem:AnisotropicLinearRescaling}
Let $\delta \leq \alpha \leq 1$ and $f \in \mathcal{S}(\R^{2k})$ with $\text{supp}(\hat{f}) \subseteq \theta_{Q_\alpha} \cap (\bigcup_{Q \in \mathcal{Q}_\delta} \theta_{Q})$, $Q_\alpha \in \mathcal{Q}_\alpha$. Then it holds
\begin{equation*}
\| f \|_{L^{p_k}(\R^{2k})} \leq C \mathcal{D}_k(\delta \alpha^{-1}) \big( \sum_{\theta \in \Theta(\delta)} \| f_\theta \|_{L^{p_k}(\R^{2k})}^2 \big)^{\frac{1}{2}}.
\end{equation*}
\end{lemma}

\section{Proof of Theorem \ref{thm:DecouplingComplexCurves}}
\label{section:ProofComplexCurves}

We turn to the proof of Theorem \ref{thm:DecouplingComplexCurves} in earnest. Above we have reduced to normalized non-degenerate curves. We shall extend the proof of Theorem \ref{thm:RealDecouplingMomentCurve} due to Guo--Li--Yung--Zorin-Kranich \cite{GuoLiYungZorinKranich2021} to complex curves. Key steps are the following:
\begin{enumerate}
\item Reduction to bilinear decoupling via Whitney decomposition,
\item Recursive estimates for asymmetric bilinear decoupling estimates via lower-dimensional decoupling,
\item Iterating the recursive estimates to conclude the proof.
\end{enumerate}

\smallskip

The first and last step transpire from the real case \cite{GuoLiYungZorinKranich2021} in a straight-forward manner, and we shall be brief. The second step hinges on two transversality properties of complex non-degenerate curves:
\begin{itemize}
\item the full transversality of the tangent vectors:
\begin{equation}
\label{eq:ComplexNondegeneracyTransversality}
|\partial_z \gamma(z) \wedge_{\C} \partial^2_z \gamma(z) \wedge_{\C} \ldots \wedge_{\C} \partial_z^k \gamma(z) | \gtrsim 1, 
\end{equation}
\item the complex transversality at separated points $z_1,z_2 \in Q_1$ with $|z_1-z_2| \gtrsim 1$:
\begin{equation}
\label{eq:ComplexTransversality}
|\partial_z \gamma(z_1) \wedge \partial_z^2 \gamma(z_1) \wedge \ldots \partial_z^{\ell} \gamma(z_1) \wedge \partial_z \gamma(z_2) \wedge \ldots \wedge \partial_z^{k-\ell} \gamma(z_2) | \gtrsim 1.
\end{equation}
\end{itemize}
The implicit constants can be chosen uniform for the class of normalized non-degenerate curves. The complex transversality properties have to be translated suitably to Euclidean geometry.


\subsection{Reduction to bilinear decoupling}
Let $\delta = 2^{-N}$. In order to use the transversality of the moment curve, we consider a Whitney decomposition of $[0,1]^2 \times [0,1]^2$ into squares of length $2^{-n}$, $n=2,\ldots,N-2$ which are denoted by $Q_{J_1^n}$, $Q_{J_2^n}$ with $\ell(Q_{J_j^n}) = 2^{-n}$ and $\text{dist}(Q_{J_1^n}, Q_{J_2^n}) \sim 2^{-n}$. The pairs $(Q_{J_1^n},Q_{J_2^n})$ at scale $2^{-n}$ are denoted by $\mathcal{W}_n$. We can write
\begin{equation*}
f^2 = \sum_{n=2}^{N-2} \sum_{(Q_{J_1^n},Q_{J_2^n}) \in \mathcal{W}_n} f_{Q_{J_1^n}} f_{Q_{J_2^n}} + \sum_{Q \in \mathcal{Q}_{4 \delta}} f^2_{Q}.
\end{equation*}
The second term, which is unseparated, can be estimated in $\| \cdot \|_{L^{\frac{p_k}{2}}}$ via Lemma \ref{lem:AnisotropicLinearRescaling}. We use Minkowski's inequality for the first term to find
\begin{equation*}
\| \sum_{n=2}^{N-2} \sum_{(Q_{J_1^n},Q_{J_2^n}) \in \mathcal{W}_n} f_{Q_{J_1^n}} f_{Q_{J_2^n}} \|_{L^{\frac{p_k}{2}}} \leq \sum_{n=2}^{N-2} \sum_{(Q_{J_1^n},Q_{J_2^n}) \in \mathcal{W}_n} \| f_{Q_{J_1^n}} f_{Q_{J_2^n}} \|_{L^{\frac{p_k}{2}}(\R^{2k})}.
\end{equation*}
This motivates the definition of bilinear decoupling constants:
\begin{definition}
Let $Q_{I_1}, Q_{I_2} \subseteq [0,1]^2$ be two squares with $\ell(Q_{I_k}) \sim 1$ and \\ $\text{dist}(Q_{I_1},Q_{I_2}) \sim 1$. Let $\gamma$ be a normalized non-degenerate curve. The bilinear decoupling constant $\mathcal{B}_k(\delta,\gamma)$  is defined as infimum over $C > 0$ such that
\begin{equation*}
\big( \int |f_{Q_{I_1}} f_{Q_{I_2}}|^{\frac{p_k}{2}} \big)^{\frac{1}{p_k}} \leq C \big( \sum_{\theta \in \Theta(\delta,Q_{I_1})} \| f_\theta \|^2_{L^{p_k}} \big)^{\frac{1}{4}} \big( \sum_{\theta \in \Theta(\delta,Q_{I_2})} \| f_\theta \|^2_{L^{p_k}} \big)^{\frac{1}{4}}.
\end{equation*}
We define 
\begin{equation*}
\mathcal{B}_k(\delta) = \sup_{\gamma: \text{normalized}} \mathcal{B}_k(\delta,\gamma).
\end{equation*}
\end{definition}
We have the following analog of rescaling the linear expression: 
\begin{lemma}
Let $Q_{J_1}$, $Q_{J_2} \subseteq [0,1]^2$ be two squares with side length $\alpha < 1/4$, which satisfy $\text{dist}(Q_{J_1},Q_{J_2}) \sim \alpha$. Then it holds:
\begin{equation*}
\big( \int \big| f_{Q_{J_1}} f_{Q_{J_2}} |^{\frac{p_k}{2}} \big)^{\frac{1}{p_k}} \leq C \mathcal{B}_k(\delta \alpha^{-1}) \big( \sum_{\theta \in \Theta(Q_{J_1},\delta)} \| f_\theta \|_{L^{p_k}}^2 \big)^{\frac{1}{4}} \big( \sum_{\theta \in \Theta(Q_{J_2},\delta)} \| f_\theta \|^2_{L^{p_k}} \big)^{\frac{1}{4}}.
\end{equation*}

\end{lemma}

Hence, we obtain from the Whitney decomposition:
\begin{lemma}
Let $\delta = 2^{-N}$. The following estimate holds:
\begin{equation*}
\mathcal{D}_k(\delta) \leq C \big( 100 + \sum_{j=2}^{N-2} \mathcal{B}_k(\delta \cdot 2^{j}) \big)^{\frac{1}{2}}.
\end{equation*}
\end{lemma}
By defining the decoupling constant as decreasing (matching the expectation that $\mathcal{D}_k(\delta)$ increases as $\delta$ goes to zero) we obtain
\begin{equation*}
\mathcal{D}_k(\delta) \leq C \log(\delta^{-1}) \mathcal{B}_k(\delta),
\end{equation*}
and it suffices to estimate bilinear expressions
\begin{equation*}
\int \big| f_{Q_{J_1}} f_{Q_{J_2}} \big|^{p_k/2} \leq C^{p_k} \big( \sum_{\theta \in \Theta(Q_{J_1},\delta)} \| f_\theta \|^2_{L^{p_k}} \big)^{\frac{p_k}{4}} \big( \sum_{\theta \in \Theta(Q_{J_2},\delta)} \| f_\theta \|_{L^{p_k}(\R^{2k})}^{2} \big)^{p_k/4}
\end{equation*}
with $\ell(Q_{J_j}) \sim 1$ and $\text{dist}(Q_{J_1},Q_{J_2}) \sim 1$.

\subsection{Locally constant property and asymmetric decoupling}

It will be convenient to consider averaged versions of the functions dual to their frequency localization. Let $f \in \mathcal{S}(\R^{2k})$ with $\text{supp}(\hat{f}) \subseteq \bigcup_{\theta \in \Theta(Q,\delta)} \theta$. Then we have that $\text{supp}(\hat{f}) \subseteq M \theta_{Q}$\footnote{$M \theta_Q$ denotes the dilation of $\theta_Q$ by $M > 1$.}, and we consider the polar set
\begin{equation*}
\mathcal{U}^{\circ}_{Q} = \{ x \in \R^{2k} : \sup_{y \in C(\theta_{Q})} |\langle x, y \rangle| \leq 1 \}
\end{equation*}
with $C(\theta_{Q}) = \theta_{Q} - F(\gamma(c_{Q}))$ denoting the shift of $Q$ to the origin.
Let $\phi_{\mathcal{U}_{I_\alpha}^o}$ denote a bump function adapted to $\mathcal{U}_{I_\alpha}^o$, which is polynomially decreasing away from $\mathcal{U}_{I_\alpha}^{o}$:
\begin{equation*}
\phi_{\mathcal{U}_{I_\alpha}^o} = |\mathcal{U}_{I_\alpha}^o|^{-1} \inf \{ t \geq 1 : x/t \in \mathcal{U}_{I_\alpha}^o \}^{-A}
\end{equation*}
with $A > 2k $ and $A \geq k (k+1)$.

\smallskip

 We can quantify the uncertainty principle as follows:
\begin{lemma}
Let $f_{Q} \in \mathcal{S}(\R^{2k})$ with $\text{supp}(\hat{f}_{Q}) \subseteq \bigcup_{\theta \in \Theta(Q,\delta)} \theta$. Then the following holds for $1 \leq p < \infty$:
\begin{equation*}
|f_{Q}|^p \leq C_p |f_{Q}|^p * \phi_{\mathcal{U}_I^o}.
\end{equation*}
\end{lemma}
Now we can consider the asymmetric expressions: Let $1 \leq \ell < k$ and $0<\alpha,\beta \leq 1$ and let $f_{Q_{a}}$, $f_{Q_{b}} \in \mathcal{S}(\R^{2k})$ with $\ell(Q_{a}) = \delta^\alpha$ and $\ell(Q_{b}) = \delta^\beta$ and for $x \in \{ a, b \}$ we suppose that
\begin{equation*}
\text{supp}(\hat{f}_{Q_{x}}) \subseteq \bigcup_{\theta \in \Theta(\delta,Q_{x})} \theta.
\end{equation*}
The asymmetric bilinear decoupling constant $\mathcal{M}_{a,b}(\delta)$ is defined as infimum over $C>0$ such that
\begin{equation*}
\begin{split}
\int \big( |f_{Q_{a}}|^{p_\ell} * \phi_{Q_{a}} \big) \big( |f_{Q_{b}}|^{p_k - p_{\ell}} * \phi_{Q_{b}} \big) d\underline{x} &\leq C^{p_k} \big( \sum_{\theta \in \Theta(Q_{a},\delta)} \| f_\theta \|^2_{L^{p_k}} \big)^{p_{\ell}/2} \\
&\quad \quad \quad \times \big( \sum_{\theta \in \Theta(Q_{b},\delta)} \| f_\theta \|^2_{L^{p_k}} \big)^{\frac{p_k-p_{\ell}}{2}}
\end{split}
\end{equation*}
the inequality holds.

\subsection{Lower-dimensional decoupling via transversality}
\label{subsection:LowerDimDecoupling}
With notations from above,
we aim to use lower-dimensional decoupling on the expression:
\begin{equation*}
\int \big( |f_{Q_{a}}|^{p_{\ell}} * \phi_{Q_{a}} \big) \big( |f_{Q_{b}}|^{p_k - p_{\ell}} * \phi_{Q_{b}} \big).
\end{equation*}

For future use, we record the following standard localization of the global decoupling:
\begin{proposition}[Global-to-local-decoupling]
Let $\gamma: \C \to \C^k$ be a normalized non-degenerate curve and $f \in \mathcal{S}(\R^{2k})$ with $\text{supp}(\hat{f}) \subseteq \mathcal{N}(\delta,\gamma)$. Then we have the following localized decoupling inequality:
\begin{equation*}
\big( \dashint_{B_{\delta^{-k}}(a)} |f|^{p_k} \big)^{\frac{1}{p_k}} \leq C \mathcal{D}_k(\delta) \big( \sum_{\theta \in \Theta(\delta)} \| f_\theta \|^2_{L_{\text{avg}}^{p_k}(w_{B_{\delta^{-k}}})} \big)^{\frac{1}{2}}.
\end{equation*}
\end{proposition}

The following proposition is the key step in the decoupling iteration:
\begin{proposition}
\label{prop:TransverseDecoupling}
Let $0<a,b<1$ and let $Q_{a}$, $Q_{b} \subseteq [0,1]^2$ be two squares of length $\delta^a$, $\delta^b$ with $\text{dist}(Q_{a},Q_{b}) \sim 1$. Define $b'=((k-\ell+1)/\ell )b$ and suppose that $a \leq b'$. Then the following estimate holds:
\begin{equation}
\label{eq:TransversalDecoupling}
\begin{split}
&\quad \int_{\R^{2k}} \big( |f_{Q_{a}}|^{p_{\ell}} * \phi_{Q_{a}} \big) \big( |f_{Q_{b}}|^{p_k - p_{\ell}} * \phi_{Q_{b}} \big) d\underline{x} \\
&\leq C \mathcal{D}_{\ell}^{p_{\ell}}(\delta^{b' \ell}) \big( \sum_{\theta \in \Theta(Q_{a},\delta^{b'})} \big[ \int \big( |f_\theta|^{p_{\ell}} * \phi_{Q_{\theta}} \big) \big( |f_{Q_{b}}|^{p_k - p_{\ell}} * \phi_{Q_{b}} \big) \big]^2 \big)^{\frac{1}{2}}.
\end{split}
\end{equation}
\end{proposition}

By this the following is immediate:
\begin{corollary}
\label{cor:BilDecEstimate}
Under the assumptions of Proposition \ref{prop:TransverseDecoupling}, we have
\begin{equation*}
\mathcal{M}_{p_{\ell},p_k-p_{\ell}}(\delta^a,\delta^b) \leq C \mathcal{D}_{\ell}(\delta^{b' \ell}) \mathcal{M}_{p_{\ell},p_k-p_{\ell}}(\delta^{b'},\delta^b).
\end{equation*}
\end{corollary}

Before we turn to the proof of Proposition \ref{prop:TransverseDecoupling}, we make transversality considerations.
The complex non-degeneracy \eqref{eq:ComplexNondegeneracyTransversality}
and the trans\-versality condition for $|z_1-z_2| \gtrsim 1$ 
%
have to be translated to $\R^{2k}$ to be utilized. We have the following lemma:
\begin{lemma}
\label{lem:RealComplexDeterminant}
Let $v_1,\ldots,v_k \in \C^k$ and $i^2 = -1$.
For $m \in \{1,\ldots,k\}$ and $j\in\{1,2\}$ let
\begin{equation*}
v^m_j = \begin{cases}
F(v_m), \quad j=1, \\
F(i \cdot v_m), \quad j=2.
\end{cases}
\end{equation*}
Then we obtain for the determinant
\begin{equation*}
| \det_{\C}(v_1,\ldots,v_k)|^2 = |\det_{\R}(v_1^1, v_2^1, \ldots, v^k_1,v^k_2)|.
\end{equation*}
\end{lemma}
\begin{proof}
For the proof we consider the sequence of complex elementary row operations which transform the complex matrix to the unit matrix. This sequence we translate into real elementary row operations of the real matrix:
Elementary row operations are
\begin{itemize}
\item swapping of rows,
\item multiplying rows with a complex number,
\item adding complex multiples of a row.
\end{itemize}
Regarding the swapping of rows we note that this leaves the modulus of the complex determinant invariant. It translates to swapping a double row, which leaves the determinant invariant.

Furthermore, multiplying a complex row with a complex number gives
\begin{equation*}
\begin{split}
&\quad (\alpha + i \beta) (a_{j1} + i b_{j1}, a_{j2} + i b_{j2}, \ldots, a_{jn} + i b_{jn}) \\
&= (\alpha a_{j1} - \beta b_{j1} + i (\beta a_{j1} + \alpha b_{j1}), \alpha a_{j2} - \beta b_{j2} + i (\beta a_{j2} + \alpha b_{j2}), \ldots ).
\end{split}
\end{equation*}
This is translated by multiplying a block matrix
\begin{equation*}
\begin{pmatrix}
\alpha & \beta & 0 & 0 & 0 & 0 & \cdots \\
- \beta & \alpha & 0 & 0 & 0 & 0 & \cdots \\
0 & 0 & 1 & 0  & 0 & 0 & \cdots \\
\vdots & \vdots & 0 & 1 & 0 & 0 & \cdots \\
\vdots & \vdots & 0 & 0 & \ddots & 0 & \cdots
\end{pmatrix}
\end{equation*}
to the $\R^{2k}$ matrix. Clearly, the modulus of the real matrix changes with the modulus square of the complex matrix.
\end{proof}

We obtain the following consequences:

\begin{lemma}
\label{lem:RealComplexNondegeneracy}
Let $\gamma : \C \to \C^n$ be a normalized non-degenerate complex curve. For any $z \in Q_1$ we have that
\begin{equation}
\label{eq:RealNonDegeneracy}
|\Gamma^1_1(z) \wedge_\R \Gamma^1_2(z) \wedge_\R \ldots \wedge_\R \Gamma_1^n(z) \wedge_{\R} \Gamma^n_2(z)| \gtrsim 1.
\end{equation}
Let $z_1,z_2 \in Q_1$ with $|z_1-z_2| \gtrsim 1$. Then it holds
\begin{equation}
\label{eq:RealTransversality}
| \Gamma^1_1(z_1) \wedge_\R \Gamma^1_2(z_1) \wedge_\R \ldots \Gamma^\ell_1(z_1) \wedge_\R \Gamma^{\ell}_2(z_1) \wedge_\R \Gamma^1_1(z_2) \wedge_\R \ldots \wedge_\R \Gamma^{n-\ell}_1(z_2) \wedge_\R \Gamma^{n-\ell}_2(z_2) | \gtrsim 1.
\end{equation}
\end{lemma}
\begin{proof}
\eqref{eq:RealNonDegeneracy} follows from applying Lemma \ref{lem:RealComplexDeterminant} to the complex non-degeneracy:
\begin{equation*}
|\partial_z \gamma(z) \wedge_\C \ldots \wedge_\C \partial_z^n \gamma(z) | \gtrsim 1.
\end{equation*}

\smallskip

With Lemma \ref{lem:RealComplexDeterminant} at hand, for the second identity it suffices to show
\begin{equation}
\label{eq:ComplexTransversalitySeparatedPoints}
|\partial_z \gamma(z_1) \wedge_{\C} \partial^2_z \gamma(z_1) \wedge_{\C} \ldots \wedge \partial_z^\ell \gamma(z_1) \wedge_{\C} \partial_z \gamma(z_2) \wedge_{\C} \ldots \wedge_{\C} \partial_z^{n- \ell} \gamma(z_2)| \gtrsim 1.
\end{equation}

To this end, we complexify the argument in \cite[Lemma~3.5]{GuoLiYungZorinKranich2021}. 
 The complex Taylor expansion of $\gamma$ around $z_1$ gives
\begin{equation*}
\partial_z^i \gamma(z_2)= \sum_{j=i}^k \frac{1}{(j-i)!} \partial_z^j \gamma(z_1) (z_2 - z_1)^{j-i} + \mathcal{O}(\partial_z^{k+1} \gamma \cdot (z_2-z_1)^{k+1-j}).
\end{equation*}
We use $\C$-multilinearity and the derivative bounds for normalized curves to obtain like in \cite{GuoLiYungZorinKranich2021}:
\begin{equation*}
\begin{split}
\text{lhs} \eqref{eq:ComplexTransversalitySeparatedPoints} &= \binom{k}{\ell} \big( \prod_{i=1}^\ell i ! \big) \big( \prod_{j=1}^{n-\ell} j ! \big) |\partial_z \gamma(z_1) \wedge_{\C} \ldots \wedge_{\C} \partial_z^{n} \gamma(z_1) | |z_2 - z_1|^{\ell (k-\ell)}\\
&\quad + \mathcal{O}(k! \cdot (2k)^{k-1} \frac{2^{-10k}}{k! (2k)^k} ) \gtrsim 1.
\end{split}
\end{equation*}

\end{proof}

We are ready for the proof of Proposition \ref{prop:TransverseDecoupling}:

\begin{proof}[Proof~of~Proposition~\ref{prop:TransverseDecoupling}]
For $z_2 = c(Q_b)_1 + i c(Q_b)_2$ we denote by \\$V^{(k-\ell)}_{\R} = \langle \Gamma^1_1(z_2),\Gamma^1_2(z_2),\ldots,\Gamma_2^{k-\ell}(z_2) \rangle$ the real tangent space of order $k-\ell$ and let
\begin{equation*}
H = \langle \Gamma^1_1(z_2), \Gamma^1_2(z_2), \ldots , \Gamma_1^{k-\ell}(z_2), \Gamma^{k-\ell}_2(z_2) \rangle^{\perp}.
\end{equation*}
By \eqref{eq:RealNonDegeneracy}, we have $H \equiv \langle \Gamma_1^{k-\ell+1}(z_2),\ldots, \Gamma_1^k(z_2), \Gamma_2^k(z_2) \rangle$ and by the uncertainty principle,
$|f_{Q_{b}}|$ is essentially constant on a scale of $\delta^{-b (k-\ell +1)} = \delta^{-b' \ell}$.

This yields the following:
\begin{equation*}
\sup_{z \in (B_{\delta^{-b' \ell}} \cap H)+y} \big( |f_{Q_{b}}|^{p_k-p_{\ell}} * \phi_{Q_{b}} \big)(z) \leq C \big( |f_{Q_{b}}|^{p_k - p_{\ell}} * \phi_{Q_{b}} \big)(y).
\end{equation*}
We write by Fubini's theorem:
\begin{equation}
\label{eq:Isolation}
\begin{split}
&\quad \int \big( |f_{Q_{a}}|^{p_\ell} * \phi_{Q_{a}} \big) \big( |f_{Q_{b}}|^{p_k - p_{\ell}} * \phi_{Q_{b}} \big) \\
&= \int_{z \in \R^{2k}} dz \dashint_{(B_{\delta^{-b' \ell}} \cap H)+z} \big[ \big( |f_{Q_a}|^{p_{\ell}} * \phi_{Q_{a}} \big)(y) \; \big( |f_{Q_{b}}|^{p_k - p_{\ell}} * \phi_{Q_{b}} \big)(y) dy \big] \\
&\lesssim \int_{z \in \R^{2k}} dz \big[ \dashint_{(B_{\delta^{-b' \ell}} \cap H)+z} \big[ \big( |f_{Q_a}|^{p_{\ell}} * \phi_{Q_{a}} \big)(y) \big] \big] \, |f_{Q_{b}}|^{p_k - p_{\ell}} * \phi_{Q_{b}} \big)(z)
\end{split}
\end{equation}
We shall apply $\ell$-dimensional (complex moment curve) decoupling on the inner integral. Let $\hat{H}^{\R} = \hat{\R}^{2k} / V^{\R}_{k-\ell}(z_2) = \hat{\R}^{2k} / H^{\perp}$ denote the Pontryagin dual of $H$ and $\pi^{\R}_{\hat{H}}:\R^{2k} \to \hat{H}$ the corresponding projection. It will be useful to observe that $\hat{H}$ can also be perceived as complex vector space: 

\smallskip

Consider the complex tangent space of order $k-\ell$ at $z_2$
\begin{equation*}
V^{(\C)}_{k-\ell}(z_2) = \langle \partial_z \gamma(z_2), \ldots, \partial_z^{k-\ell} \gamma(z_2) \rangle_\C, \text{ and } \hat{H}^{\C} = \hat{\C}^{k} / V^{(\C)}_{k-\ell}(z_2).
\end{equation*}
Since $F(V^{(\C)}_{k-\ell}(z_2)) = V^{(\R)}_{k-\ell}(z_2)$, we have $F(\hat{H}^{\C}) = \hat{H}^{\R}$. We remark that as a consequence of Lemma \ref{lem:RealComplexNondegeneracy} we have
\begin{equation}
\label{eq:QuantitativeTransversality}
\begin{split}
\hat{\R}^{2k} &= V_{\ell}^{(\R)}(z_1) + V^{(\R)}_{k-\ell}(z_2), \\
\hat{\C}^{k} &= V_{\ell}^{(\C)}(z_1) + V^{(\C)}_{k-\ell}(z_2)
\end{split}
\end{equation}
in a quantitatively transverse sense, i.e., for $a \in V^{(\K)}_{\ell}(z_1)$, $b \in V^{(\K)}_{k-\ell}(z_2)$, we have
\begin{equation*}
|a \wedge_{\K} b| \gtrsim_{|z_1-z_2|} 1.
\end{equation*}

\medskip

To apply lower-dimensional decoupling on the averaged expression in \eqref{eq:Isolation}, we verify the following two claims:
\begin{enumerate}
\item $\pi_{\hat{H}^{\C}} \gamma \big\vert_{Q_a}$ is a non-degenerate complex curve,
\item $\text{supp}(\widehat{f \big\vert_{H+z}}) \subseteq \mathcal{N}(M \delta^a, \pi_{\hat{H}^{\C}} \gamma)$.
\end{enumerate}

\emph{Ad 1):} Change the basis in $\hat{\C}^k$ to $\{ \partial_z \gamma(z_1), \ldots, \partial_z^{\ell} \gamma(z_1), \partial_z \gamma(z_2), \ldots, \partial_z^{k-\ell} \gamma(z_2) \}$ with modulus of the determinant invoking Lemma \ref{lem:RealComplexNondegeneracy} given by
\begin{equation*}
|\partial_z \gamma(z_1) \wedge_{\C} \ldots \wedge_{\C} \partial_z^{\ell} \gamma(z_1) \wedge_{\C} \partial_z \gamma(z_2) \ldots \wedge_{\C} \partial_z^{k-\ell} \gamma(z_2) | \geq c(|z_1-z_2|) \gtrsim 1.
\end{equation*}
This implies
\begin{equation*}
|\partial_z \pi_{\hat{H}^{\C}} \gamma(z_1) \wedge_{\C} \ldots \wedge_{\C} \partial_z^{\ell} \pi_{\hat{H}^{\C}} \gamma(z_1) | \geq c(|z_1-z_2|) \gtrsim 1.
\end{equation*}

With $\gamma$ being a normalized non-degenerate complex curve, by breaking the support, Taylor expansion and affine transformation, $\pi_{\hat{H}^{\C}} \gamma$ can be decomposed into a fixed number of normalized non-degenerate complex curves.

\medskip

\emph{Ad 2):} We will prove that for $Q(\gamma) \subseteq Q_a$ it holds
\begin{equation}
\label{eq:RectangleProjection}
\pi_{\hat{H}^{\R}}(\theta_{Q(\gamma)}) \subseteq M \theta_{Q(\pi_{\hat{H}^{\C}} \gamma)},
\end{equation}
from which the claim follows.

To verify this, we perceive $\theta_{Q(\gamma)}$ as complex anisotropic rectangle. Let $z_1 = c(Q(\gamma))_1 + i c(Q(\gamma))_2$ denote the center of $Q(\gamma)$ perceived as complex number, let
\begin{equation*}
\theta^{\C}_{Q(\gamma)} = \{ \gamma(z_1) + a_1 \partial_z \gamma(z_1) + \ldots + a_k \partial_z^{k} \gamma(z_1) \, : \; a_i \in \C, \, |a_i| \leq 4 \ell(Q(\gamma))^i \},
\end{equation*}
and observe that $F(\theta^{\C}_{Q(\gamma)}) = \theta_{Q(\gamma)}$. 

For the first $\ell$ complex sides of $\theta^{\C}_{Q(\gamma)}$ we observe that 
\begin{equation*}
(\partial^i_z \pi_{\hat{H}^{\C}} \gamma(z_1))_{i=1,\ldots,\ell} = ( \pi_{\hat{H}^{\C}} \partial^i_z \gamma(z_1))_{i=1,\ldots,\ell} =  (\partial^i_z \gamma(z_1) + V^{(k-\ell)}_{\C}(z_2) )_{i=1,\ldots,\ell}
\end{equation*}
 is a basis of $\hat{\C}^k / V^{(k-\ell)}_{\C}(z_2)$, which is immediate from the quantitative transversality.
 
 We compute the effect of $\pi_{\hat{H}^{\C}}$ on the remaining $k-\ell$ sides: To this end, we write
\begin{equation*}
\partial_z^{\ell+1} \gamma(z_1) = \sum_{i=1}^{\ell} a_i \partial_z^i \gamma(z_1) + \sum_{j=1}^{k-\ell} b_j \partial_z^j \gamma(z_2)
\end{equation*}
and we shall see that $|a_i| \lesssim 1$. To simplify notations, we exemplarily estimate $a_1$. We take the wedge product of the above identity with 
\begin{equation*}
\partial^2_{z} \gamma(z_1) \wedge_{\C} \partial^3_z \gamma(z_1) \ldots \wedge_{\C} \partial^{\ell}_z \gamma(z_1) \wedge_{\C} \partial_z \gamma(z_2) \wedge_{\C} \ldots \wedge_{\C} \partial_z^{k-\ell} \gamma(z_2)
\end{equation*}
to find
\small
\begin{equation*}
|a_1| \leq \frac{|\partial_z^{\ell+1} \gamma(z_1) \wedge_{\C} \partial^2_{z} \gamma(z_1) \wedge_{\C} \partial^3_z \gamma(z_1) \ldots \wedge_{\C} \partial^{\ell}_z \gamma(z_1) \wedge_{\C} \partial_z \gamma(z_2) \wedge_{\C} \ldots \wedge_{\C} \partial_z^{k-\ell} \gamma(z_2)| }{ |\partial_z \gamma(z_1) \wedge_{\C} \partial_z^2 \gamma(z_1) \ldots \wedge_{\C} \partial_z^{\ell} \gamma(z_1) \wedge_{\C} \partial_z \gamma(z_2) \ldots \wedge_{\C} \partial_z^{k-\ell} \gamma(z_2)|}.
\end{equation*}
\normalsize
Then, it is a consequence of the quantitative transversality and boundedness of the derivatives that $|a_1| \lesssim 1$.


\medskip

With 1) and 2) at hand, we obtain
\begin{equation*}
\dashint_{H \cap B_{\delta^{-b' \ell}}(z)} |f_{Q_{a}}|^{p_\ell} * \phi_{Q_{a}} \leq C \mathcal{D}_{\ell}^{p_\ell} (\delta^{b' \ell}) \big( \sum_{\theta \in \Theta(\delta)} \| f_\theta \|^2_{L^{p_{\ell}}(\phi_{H \cap \delta^{-b' \ell}})} \big)^{p_{\ell}/2},
\end{equation*}
which plugged into \eqref{eq:Isolation} yields the claim.

%
\end{proof}

This extends the notion of full lower-dimensional transversality to the complex case.

\subsection{Bilinear decoupling iteration}

The proof of Theorem \ref{thm:DecouplingComplexCurves} is concluded following the bilinear decoupling iteration from \cite[Section~4]{GuoLiYungZorinKranich2021}. We shall be brief.

 The bedrock of the decoupling iteration is the inequality (see Corollary \ref{cor:BilDecEstimate}):
\begin{equation*}
\mathcal{M}_{p_{\ell},p_k-p_{\ell}}(\delta^a,\delta^b) \leq C \mathcal{D}_{\ell}(\delta^{b' \ell}) \mathcal{M}_{p_{\ell},p_k-p_{\ell}}(\delta^{b'},\delta^b)
\end{equation*}
with $b'=((k-\ell+1)/\ell) b$.

\smallskip

A simple, but important relation is established by H\"older's inequality (see \cite[Lemma~4.1]{GuoLiYungZorinKranich2021}): For $\ell \in \{1,\ldots,k-1 \}$, if $a,b\in(0,1)$ and $\delta \in (0,1)$, then
\begin{equation*}
\mathcal{M}_{p_{\ell},p_k-p_{\ell}}(\delta^a,\delta^b) \leq \mathcal{M}_{p_k-p_{\ell},p_{\ell}}(\delta^b,\delta^a)^{\frac{1}{k-\ell+1}} \mathcal{M}_{p_{\ell-1},p_k-p_{\ell-1}}(\delta^a,\delta^b)^{\frac{k-\ell}{k-\ell+1}}.
\end{equation*}

Now we induct on dimension. Since $k=1$ is immediate from Plancherel's theorem, suppose that $k > 1$ and Theorem \ref{thm:DecouplingComplexCurves} has been proved for all lower values of $k$.
Combining the above displays like in \cite[Lemma~4.2]{GuoLiYungZorinKranich2021} and iteration like in \cite[Section~4]{GuoLiYungZorinKranich2021} will finish the proof of
$\mathcal{D}_k(\delta) \lesssim_\varepsilon \delta^{-\varepsilon}$ and the proof of Theorem \ref{thm:DecouplingComplexCurves}.

$\hfill \Box$

\section{Logarithmic improvement of decoupling for the cubic moment curve}

In this section we show the logarithmic sharpening of the decoupling inequality for the cubic (real) moment curve. Although the qualitative arguments were presented above, to carefully keep track of the constants we provide details. Additional input comes from the improved decoupling for the non-degenerate curves in $1+1$-dimensions due to \cite{GuthMaldagueWang2024} and quantifying the decoupling iteration for the cubic moment curve.

\label{section:LogarithmicImprovement}
\subsection{Preliminaries}
\label{subsection:Preliminaries}
\subsubsection{Notations}
For $k \in \N$ in the following let $\gamma_k:[0,1] \to \R^k$, $(\gamma_k(t))_m= t^m/m!$ denote the moment curve in $\R^k$. Let $\delta \in \N^{-1} = \{ \frac{1}{n} : n \in \N \}$. For a closed interval $[a,b] = I \subseteq [0,1]$ with $|I| \delta^{-1} \in \N$ we denote by $\mathcal{I}(I,\delta)$ the decomposition into closed intervals of length $\delta$: $I = \bigcup_{j=0}^{N-1} [a+j\delta, a+ (j+1) \delta]$ with $N \delta = |I|$. If $I = [0,1]$, we let $\mathcal{I}(I,\delta) = \mathcal{I}(\delta)$. We define for a curve $\gamma:[0,1] \to \R^k$ the parallelepiped $\theta_{J,\gamma}$ with center $c_J$ of dimensions $4 |J| \times 4 |J|^2 \times \ldots \times 4 |J|^k$ into directions $\partial \gamma(c_J), \ldots, \partial^k \gamma(c_J)$. When $\gamma = \gamma_k$, we simply write $\theta_{J}$.

\smallskip

In the following, for an interval $J$ and curve $\gamma$, let $\mathcal{U}_{J,\gamma}^o$ denote the parallelepiped centered at the origin, which is dual to $\theta_{J,\gamma}$, that is
\begin{equation*}
\mathcal{U}_{J,\gamma}^o = \{ x \in \R^k \, : \, \big| \langle x, \partial^i \gamma(c_J) \rangle \big| \leq \frac{1}{4} |J|^{-i}, \quad 1 \leq i \leq k \}.
\end{equation*}
This is a parallelepiped of size $\sim |J|^{-1} \times |J|^{-2} \times |J|^{-3}$. We define a bump function adapted to $\mathcal{U}_J^o$ by
\begin{equation*}
\phi_J(x) = |\mathcal{U}_J^o|^{-1} \inf \{ t \geq 1 \, : x/t \in \mathcal{U}_J^o \}^{-A}
\end{equation*}
with $A = \max( 10k, k (k+1)/2)+1$.

\smallskip

$\phi_J$ is $L^1$-normalized as can be seen from anisotropic dilation: $\int_{\R^k} \phi_J(x) dx \leq C_{4,k}$.
For $k \in \N$, we define the critical decoupling exponent for the moment curve $\Gamma_k$ as $p_k = k(k+1)$.
In this section we define the decoupling constant $\mathcal{D}_{k}(\delta)$ for $\gamma_k$ as monotone decreasing in $\delta$ (this means if $\delta^{-1}$ becomes larger, the decoupling constant is also supposed to become larger) and smallest constant, which satisfies
\begin{equation*}
\big\| \sum_{I \in \mathcal{I}(\delta)} f_I \big\|_{L^{p_k}(\R^k)} \leq \mathcal{D}_{k}(\delta) \big( \sum_{I \in \mathcal{I}(\delta)} \| f_I \|_{L^{p_k}(\R^k)}^2 \big)^{\frac{1}{2}}.
\end{equation*}
To ease notation, in this section we denote the decoupling constant for the real moment curve as well with $\mathcal{D}_k$.

\subsubsection{Bilinear reduction, and uncertainty principle}

We define the bilinear decoupling constant $\mathcal{B}_{k}(\delta)$ as smallest constant decreasing in $\delta$ such that
\begin{equation*}
\begin{split}
\big( \int_{\R^k} |\sum_{J_1 \in \mathcal{P}(I_1,\delta)} f_{J_1}|^{p_k/2} |\sum_{J_2 \in \mathcal{P}(I_2,\delta)} f_{J_2}|^{p_k/2} \big)^{1/p_k} &\leq \mathcal{B}_{k}(\delta) \big( \sum_{J_1 \in \mathcal{I}(I_1,\delta)} \| f_{J_1} \|^2_{L^{p_k}} \big)^{1/4} \\
&\quad \times  \big( \sum_{J_2 \in \mathcal{I}(I_2,\delta)} \| f_{J_2} \|^2_{L^{p_k}} \big)^{1/4}.
\end{split}
\end{equation*}
Record the linear-to-bilinear reduction:
\begin{lemma}[{Bilinear~reduction,~\cite[Lemma~2.2]{GuoLiYungZorinKranich2021}}]
If $\delta = 2^{-M}$, then there is $C_1 > 0$ such that
\begin{equation*}
\mathcal{D}_{k}(\delta) \leq C_1 \big( 1+ \sum_{n=2}^M \mathcal{B}_{k}(2^{-M+n-2})^2 \big)^{1/2}.
\end{equation*}
\end{lemma}

Record submultiplicativity by affine rescaling:
\begin{lemma}
We have for $\delta$, $\sigma$, $\delta/32 \sigma \in \N^{-1}$:
\begin{equation}
\label{eq:SubMultiplicativity}
\mathcal{D}_k(\delta) \leq \mathcal{D}_k(\sigma) \mathcal{D}_k(\delta/32 \sigma).
\end{equation}
\end{lemma}
%

In the iteration to estimate $\mathcal{D}_{k}(\delta)$, we use monotonicity of $\mathcal{B}_{k}(\delta)$ to write
\begin{equation}
\label{eq:LinearBilinearReduction}
\mathcal{D}_{k}(\delta) \leq C_1 \log(\delta^{-1}) \mathcal{B}_{k}(\delta).
\end{equation}
The reason we do not resort to the slightly sharper argument of broad-narrow reduction, which is used by Li \cite{Li2021}, is that the unit distance separation of the intervals simplifies the forthcoming arguments, and we are losing logarithmic factors in the iteration anyway.


\medskip

We also use the following instance of uncertainty prinicple:
\begin{lemma}[{\cite[Lemma~3.3]{GuoLiYungZorinKranich2021}}]
For $p \in [1,\infty)$ and $J \subseteq [0,1]$ we have
\begin{equation*}
|g_J|^p \leq C_p  (|g_J|^p * \phi_J),
\end{equation*}
for every $g_J$ with $\text{supp} (\hat{g}_J) \subseteq C' \mathcal{U}_J$.
\end{lemma}

Record the following trivial bound due to Cauchy-Schwarz:
\begin{equation}
\label{eq:TrivialDecoupling}
\mathcal{D}_k(\delta) \leq \delta^{-\frac{1}{2}}.
\end{equation}

We end with an overview of constants:

\subsubsection{Overview of constants}

In the following we denote by $C_i$, $i=1,\ldots,8$, $c$ fixed (possibly very large) constants, which will be defined in the course of the argument.
\begin{itemize}

\item $C_1$ is used in the linear-to-bilinear reduction \eqref{eq:LinearBilinearReduction},
\item $C_2$, $C_3$ are used to record constants in the arguments involving lower dimensional decoupling, $c \geq 1$ denotes the exponent in the logarithmic loss for the $\ell^2 L^6$-decoupling,
\item $C_4$ depends on the $L^1$-norm of an essentially $L^1$-normalized function (see Lemma \ref{lem:HoelderInequalityI}),
\item $C_5$ is a constant, which comes up in the key iteration to lower the scale,
\item $C_6$ comes from an application of the triangle inequality to lower the scale once to $\nu$ (see \eqref{eq:TrivialDecouplingII}),
\item $N=N(\varepsilon,c)$ will later in the proof denote the number of iterations to lower the scale,
\item $C_7$, $C_8$ are absolute constants used to record intermediate estimates for $\mathcal{D}_3(\delta)$ after carrying out the decoupling iteration (see Lemmas \ref{lem:ChoiceNMomentCurve}, \ref{lem:IntermediateResultIIMomentCurve}).
\end{itemize}

\subsection{Decoupling in one and two dimensions}
\label{subsection:DecouplingLowDimensions}
In this section we argue how the improved decoupling result by Guth--Maldague--Wang extends to the family of curves presently considered. Some of the arguments are already contained in \cite[Appendix]{GuthMaldagueWang2024}. In the following we quantify the constants. Then we shall see how we can use lower dimensional decoupling in bilinear expressions.
 Their improved decoupling result is formulated for normalized curves as follows:
\begin{theorem}[{\cite[Appendix]{GuthMaldagueWang2024}}]
\label{thm:ImprovedDecouplingParaboloid}
Let $\gamma : [0,1] \to \R^2$, $\gamma(t) = (t,h(t))$ be a curve such that $h \in C^2([-1,2])$ with $h(0) = h'(0) = 0$ and $\frac{1}{2} \leq h''(t) \leq 2$. Then, there are $C,c>0$ such that for $(f_J)_{J \in \mathcal{I}_\delta}$ with $\text{supp} (\hat{f}_J) \subseteq \theta_{J,\gamma}$ we have:
\begin{equation*}
\big\| \sum_{J \in \mathcal{I}_{\delta}} f_J \big\|_{L^6(\R^2)} \leq C (\log (\delta^{-1}))^c \big( \sum_{J} \| f_J \|_{L^6(\R^2)}^2 \big)^{\frac{1}{2}}.
\end{equation*}
\end{theorem}
In the following we want to argue that the result extends to more general curves $\gamma(t) = (\gamma_1(t),\gamma_2(t)) \in C^5$ with 
\begin{equation}
\label{eq:TorsionGamma}
\| \gamma \|_{C^5} \leq D_3 < \infty \text{ and } 0< D_1 \leq | \gamma'(t) \wedge \gamma''(t) | \leq D_2 < \infty.
\end{equation}
\begin{proposition}[Stability~of~improved~decoupling]
\label{prop:StabilityImprovedDecoupling}
Suppose $\gamma \in C^5$ satisfies \eqref{eq:TorsionGamma}, and let $(f_J)_{J \in \mathcal{I}_\delta(I)}$ with $\text{supp} (\hat{f}_J) \subseteq C' \theta_{J,\gamma}$. Then, there is $C(\underline{D},C')$ such that
\begin{equation}
\label{eq:StableDecoupling}
\big\| \sum_J f_J \big\|_{L^6(\R^2)} \leq C (\log(\delta^{-1}))^c \big( \sum_J \| f_J \|_{L^6}^2 \big)^{\frac{1}{2}}.
\end{equation}
\end{proposition}
\begin{proof}
In the first step we reduce the curves $\gamma$ to $(t,h(t))$ by finite decomposition, rotation, and translation, which only depends on $\underline{D}$:
For any point $\gamma(t_*)$ we can obtain by rotation and translation that $\gamma(t_*) = 0$, $\dot{\gamma}(t_*) = (c,0)$ for some $c>0$, and $\ddot{\gamma}(t_*) > 0$. By the implicit function theorem we obtain a reparametrization $t=g(s)$ such that $\gamma_1(g(s)) = s$. The interval on which the reparametrization exists depends on $c$ and $\| \gamma \|_{C^2}$. $c$ is bounded from above by $\mathcal{D}_3$ and from below by using the torsion:
\begin{equation*}
\begin{vmatrix}
c & \ddot{\gamma}_1(t_*) \\
0 & \ddot{\gamma}_2(t_*)
\end{vmatrix}
= c |\ddot{\gamma}_2(t_*)| \geq D_1 \Rightarrow c \geq \frac{D_1}{D_3}.
\end{equation*}
This means we find finitely many curves $\tilde{\gamma}(s) = (s,h(s))$ with $h(0) = h'(0) = 0$ and $0<D_1' \leq h''(s) \leq D_2' < \infty$ with $D_i' = D_i'(\underline{D})$. We can compare the rectangles $\tilde{\gamma}$ and $\gamma$ by noting that from $\tilde{\gamma}(s) = \gamma(g(s))$ follows:
\begin{equation*}
\dot{\tilde{\gamma}}(s) = \dot{\gamma}(g(s)) g'(s), \quad \ddot{\tilde{\gamma}}(s) = \ddot{\gamma}(g(s)) (g'(s))^2 + \dot{\gamma}(g(s)) g''(s).
\end{equation*}
The bilipschitz comparability of rectangles follows then from $g'(s) \sim_{\underline{D}} 1$ and $|g''(s)| \leq \kappa(\underline{D})$. In $s$ parametrization, the rectangles $C' \theta_{J,\gamma}$ become centered at $\gamma(t_J) = \tilde{\gamma}(s_J)$ and can be contained in rectangles of length $C'' \delta \times C'' \delta^2$ in the directions $\dot{\tilde{\gamma}}(s_J)$, $\ddot{\tilde{\gamma}}(s_J)$. For this reason we observe $\text{supp} (\hat{f}_J) \subseteq \mathcal{C}'' \theta_{J,\tilde{\gamma}}$.

In the following $\gamma(t) = (t,h(t))$, and we turn to normalization of $h$. We subdivide $[0,1]$ into intervals $I_s$ of length $s$. A Taylor expansion of $\gamma$ at the center $t_c$ gives
\begin{equation*}
\gamma(t) = \gamma(t_c) + \gamma'(t_c) (t-t_c) + \gamma''(t_c) \frac{(t-t_c)^2}{2} + O((t-t_c)^3).
\end{equation*}
Since $|\dot{\gamma}(t_c) \wedge \ddot{\gamma}(t_c)| = |h''(t_c)| \neq 0$, there is an anisotropic dilation $\underline{d}=\text{diag}(d_1,d_2)$ such that after translation
\begin{equation*}
\tilde{\gamma}(t) = t e_1+ \frac{t^2}{2} e_2 + G(t) t^3 e_2.
\end{equation*}
In the above display $e_i$ denote the unit vectors such that $\gamma(t) = \gamma_1(t) e_1 + \gamma_2(t) e_2 = (\gamma_1(t),\gamma_2(t))$.
The representation $G(t) t^3 e_2$ with $G \in C^2$ for the third order remainder term in the Taylor expansion (after dilation) follows from the integral representation of the remainder:
\begin{equation*}
R_3(t) = \int_0^t \frac{\gamma^{(3)}(s)}{6} (t-s)^3 ds.
\end{equation*}
We obtain
\begin{equation*}
G(t) = \int_0^t \frac{\gamma^{(3)}(s)}{6} (1-\frac{s}{t})^3 ds = t \int_0^1 \frac{\gamma^{(3)}(ts')}{6} (1-s')^3 ds'
\end{equation*}
and for $\gamma \in C^5$ we find $G \in C^2$ and $\| G \|_{C^2} \leq \kappa(\underline{D})$. Moreover, 
\begin{equation*}
\tilde{\gamma}''(t) = e_2 + (G'(t) t^3 + 3 G(t) t^2)' e_2 = e_2 + (G''(t) t^3 + 6 t^2 G'(t) + 6 G(t) t)e_2.
\end{equation*}
Clearly, $|G''(t) t^3 + 6 t^2 G'(t) + 6 G(t) t| = O_{\underline{D}}(t)$ and choosing $s$ small enough only depending on $\underline{D}$, we finish the decomposition into curves of the kind $\gamma(t) = (t,h(t))$ with $h(0) = h'(0) = 0$ and $\frac{1}{2} \leq h''(t) \leq 2$ up to rigid motions and anisotropic dilations controlled by $\underline{D}$. Now we consider the decoupling of $(t,h(t))$ with $\text{supp} \hat{f}_J \subseteq C \mathcal{U}_{J,\gamma}$ and shall prove that
\begin{equation*}
\big\| \sum_{J \in \mathcal{I}(\delta)} f_J \big\|_{L^6(\R^2)} \leq \tilde{C}(C,C') (\log(\delta^{-1}))^c \big( \sum_J \| f_J \|_{L^6(\R^2)}^2 \big)^{\frac{1}{2}}
\end{equation*}
with $C$ like in Theorem \ref{thm:ImprovedDecouplingParaboloid}. First, we observe that Theorem \ref{thm:ImprovedDecouplingParaboloid} applies with $\tilde{C} = C$ for $C' \leq 1$. We turn to $C' \geq 1$: The minor technical issue is that the blocks $\theta_{J,\gamma}$ are overlapping more often than in the original collection. We observe that these blocks are in the $\tilde{\delta}$-neighbourhood for $\tilde{\delta} = 10^{10} (C')^2 \delta$. So we can apply decoupling for $\tilde{\delta}$, but the decomposition into $\theta_{\tilde{J},\gamma}$ for $\tilde{J} \in \mathcal{I}_{\tilde{\delta}}$ is too coarse. For $\tilde{J}$ we choose a collection $\mathcal{J}$ of intervals $J \subseteq \tilde{J}$ such that $\sum_{\tilde{J}} f_{\tilde{J}} = \sum_{J} f_J$ and find:
\begin{equation*}
\big\| \sum_{\tilde{J} \in \mathcal{I}(\tilde{\delta})} f_{\tilde{J}} \big\|_{L^6(\R^2)} \leq C \log((10^{10} (C')^2 \delta)^{-1})^c \big( \sum_{\tilde{J} \in \mathcal{I}(\tilde{\delta})} \| f_{\tilde{J}} \|_{L^6(\R^2)}^2 \big)^{\frac{1}{2}}. 
\end{equation*}
Since $\# \{ J \subseteq \tilde{J} \} = \mathcal{O}((C')^2)$, an application of Cauchy-Schwarz finishes the proof:
\begin{equation*}
\big\| \sum_{J \in \mathcal{U}_J} f_J \big\|_{L^6(\R^2)} \leq \tilde{C}(C,C') (\log(\delta^{-1}))^c \big( \sum_{J} \| f_J \|^2_{L^6(\R^2)} \big)^{\frac{1}{2}}.
\end{equation*}
\end{proof}
We summarize uniform decoupling inequalities for families of curves: Suppose $\ell \in \{1,2\}$ and $\gamma:[0,1] \to \R^\ell$ is a curve such that
\begin{equation}
\label{eq:AssumptionTorsion}
\| \gamma \|_{C^5} \leq D_3 \text{ and for any } t \in [0,1]: \, D_1 \leq \big| \bigwedge_{i=1}^{\ell} \partial^i \gamma(t) \big| \leq D_2.
\end{equation}

\begin{proposition}[Decoupling~for~curves~with~torsion~for~$d=1,2$]
\label{prop:UniformDecoupling}
Suppose that $\ell \in \{1,2\}$, and $\gamma:[0,1] \to \R^\ell$ is a curve satisfying \eqref{eq:AssumptionTorsion}. Then, for any $C>0$, any $\delta \in (0,1)$, and any tuple of functions $(f_J)_{J \in \mathcal{I}(\delta)}$ with $\text{supp}(\hat{f}_J) \subseteq C \theta_{J,\gamma}$ for any $J$, the following inequality holds:
\begin{align}
\big\| \sum_{J \in \mathcal{I}(\delta)} f_J \big\|_{L^{p_\ell}(\R^{\ell})} &\leq C_{\ell}'(C,\underline{D},\delta) \big( \sum_{J \in \mathcal{I}(\delta)} \| f_J \|^2_{L^{p_\ell}(\R^{\ell})} \big)^{1/2} 
\end{align}
with
\begin{equation*}
C_{\ell}'(C,\underline{D},\delta) = 
\begin{cases}
C'(C,\underline{D}), &\quad \ell = 1, \\
C'(C,\underline{D}) (\log(\delta^{-1}))^c, &\quad \ell =2.
\end{cases}
\end{equation*}
\end{proposition}
\begin{proof}
For $\ell =1$ this is obvious, for $\ell = 2$ this is Proposition \ref{prop:StabilityImprovedDecoupling}.
\end{proof}
\begin{corollary}
\label{cor:LocalImprovedDecoupling}
Under the assumptions of Proposition \ref{prop:UniformDecoupling}, for every ball $B \subseteq \R^{\ell}$ of radius $\delta^{-\ell}$, we have
\begin{align}
\dashint_B \big| \sum_{J \in \mathcal{I}(\delta)} f_J \big|^{p_\ell} \leq C_{\ell}'(C,\underline{D},\delta) \big( \sum_{J \in \mathcal{I}(\delta)} \| f_J \|^2_{L^{p_{\ell}}(\phi_B)} \big)^{p_{\ell}/2}.
\end{align}
In the above display $\phi_B(x) = |B|^{-1} (1+ \delta^{\ell} \text{dist}(x,B))^{-30}$ denotes an $L^1$-normalized bump function adapted to $B$, and $\dashint_{B}$ denotes the average integral.
\end{corollary}
\begin{proof}
We apply Proposition \ref{prop:UniformDecoupling} to functions $f_J \psi_B$, where $\psi_B$ is a Schwartz function such that $|\psi_B| \gtrsim 1$ on $B$ and $\text{supp} (\hat{\psi}_B) \subseteq B(0,\delta^{\ell})$. For $B$ centered at the origin, it suffices to consider $\psi_B(x)= \delta^{-\ell^{\ell}} \int_{\R^3} e^{ix.\xi} a(\delta^{-\ell} \xi) d\xi$ with $a \in C^\infty_c(\R^3)$ a radially decreasing function satisfying $a(0) = 1$, $a \geq 0$ and having support in $B(0,c)$ for $c$ small enough. The general case follows from translation. Then
\begin{equation*}
\begin{split}
\dashint_{B} \big| \sum_J f_J \big|^{p_\ell} &\lesssim \int_{\R^3} \big| \sum_J f_J \frac{\psi_B}{|B|^{1/p_{\ell}}} \big|^{p_{\ell}} \lesssim \int_{\R^3} \big| \sum_J f_J \frac{\psi_B}{|B|^{1/p_{\ell}}} \big|^{p_{\ell}} \\
&\leq C'(C,\underline{D}) C_{\ell}(\delta) \big( \sum_J \| \frac{f_J \psi_B}{|B|^{1/p_{\ell}}} \|^2_{L^{p_{\ell}}} \big)^{p_{\ell}/2}
\end{split}
\end{equation*}
with
\begin{equation*}
C_{\ell}(\delta) = 
\begin{cases}
1, &\quad \ell = 1, \\
(\log(\delta^{-1}))^c, &\quad \ell = 2.
\end{cases}
\end{equation*}

To conclude the proof, we need to argue that
\begin{equation*}
\sup_x \frac{|\psi_B(x)|^{p_{\ell}}}{|B| \phi_B(x)} \lesssim 1.
\end{equation*}
This follows from the rapid decay of $\psi_B$ away from $B$ (which is by our definition of $\psi_B$ faster than any polynomial). Let $R = \delta^{-\ell}$ and suppose again $B$ is centered at the origin. For $\text{dist}(x,B) \leq R$ we have $|\psi_B(x)|^{p_{\ell}} \leq C$, $|B| \phi_B(x) \geq 1$. For $\text{dist}(x,B) \geq R$, which means $x \geq 2R$, we can estimate
\begin{equation*}
\psi_B(x) = \psi_1(R^{-1} x) \leq C_N (1+R^{-1} x)^{-N}.
\end{equation*}
 Therefore, $(\psi_1(R^{-1} x))^{p_{\ell}} \leq C_N (R^{-1} x)^{-N \cdot p_\ell}$ and $|B| \phi_B(x) \sim (R^{-1} x)^{-30}$. Choosing $N$ large enough yields an acceptable contribution.
\end{proof}

\begin{lemma}[Lower~degree~decoupling~(see~{\cite[Lemma~3.5.]{GuoLiYungZorinKranich2021}})]
\label{lem:LowerDegreeDecoupling}
Let $\ell \in \{ 1,2 \}$. Let $\delta \in (0,1)$ and $(f_K)_{K \in \mathcal{I}(\delta)}$ be a tuple of functions so that $\text{supp} \hat{f}_K \subseteq \theta_K$ for every $K$. If $0 \leq a \leq (3-\ell + 1)b/\ell$, then for any pair of intervals $I \in \mathcal{I}(\delta^a)$, $I' \in \mathcal{I}(\delta^b)$ with $\text{dist}(I,I') \geq 1/4$, we obtain for $\ell = 1$:
\begin{equation}
\label{eq:LowerDimensionalDecouplingl1}
\int_{\R^3} (|f_I|^2 * \phi_I) (|f_{I'}|^{10} * \phi_{I'}) \leq C_1 \sum_{J \in \mathcal{I}(I,\delta^{3b})} \int_{\R^3} (|f_J|^2 * \phi_J) (|f_{I'}|^{10} * \phi_{I'}),
\end{equation}
and for $\ell = 2$:\small
\begin{equation}
\label{eq:LowerDimensionalDecouplingl2}
\int_{\R^3} (|f_I|^6 * \phi_I) (|f_{I'}|^6 * \phi_{I'}) \leq C_2 (\log (\delta^{-b}))^c \big( \sum_{J \in \mathcal{I}(I,\delta^b)} \big( \int_{\R^3} \big( |f_J|^6 * \phi_J \big) \big( |f_{I'}|^6 * \phi_{I'} \big) \big)^{\frac{1}{3}} \big)^3.
\end{equation}
\normalsize
\end{lemma}

Now we are ready to prove Lemma \ref{lem:LowerDegreeDecoupling}. We revisit the argument from \cite{GuoLiYungZorinKranich2021} to make the involved implicit constants transparent. By quantifying the decoupling constants, we can improve the $\delta^{-\varepsilon}$ bound from \cite{GuoLiYungZorinKranich2021} as claimed.
\begin{proof}[Proof~of~Lemma~\ref{lem:LowerDegreeDecoupling}]
Denote $b' = (3-\ell+1)b/\ell$ and $k = 3$.
Fix $\xi' \in I'$, let $V^{m}(\xi') = \text{span} (\partial^1 \gamma_k(\xi'), \ldots, \partial^m \gamma_k(\xi'))$ be the tangent space for $m \in \{1,2,3\}$, and let $\hat{H} = \R^k / V^{k-\ell}(\xi')$ be the quotient space. Let $\pi: \R^k \to \hat{H}$ be the projection onto $\hat{H}$. For every $\xi \in I$, we have like in the proof of Proposition \ref{prop:TransverseDecoupling}:
\begin{equation}
\label{eq:LowerDimDecouplingAssumptions}
|\partial^1 (\pi \circ \gamma_k)(\xi) \wedge_{\R} \ldots \wedge_{\R} \partial^{\ell} (\pi \circ \gamma_k)(\xi)| \gtrsim 1, \quad \pi(\theta_J) \subseteq C' \theta_{J,\pi \circ \gamma_k}.
\end{equation}
By Fubini's theorem
\begin{equation}
\label{eq:Fubini}
\int_{\R^k} \big( |f_I|^{p_{\ell}} * \phi_I \big) \big( |f_{I'}|^{p_k - p_{\ell}} * \phi_{I'} \big) 
\lesssim \int_{z \in \R^k} \big( \dashint_{B_H(z,\delta^{-b' \ell})} |f_I|^{p_\ell} * \phi_I \big) \big( |f_{I'}|^{p_k - p_{\ell}} * \phi_{I'} \big)(z),
\end{equation}
where $B_H(z,\delta^{-b' \ell})$ is the $\ell$-dimensional ball with radius $\delta^{-b' \ell}$ centered at $z$ inside the affine subspace $H+z$. 

By \eqref{eq:LowerDimDecouplingAssumptions} we can use lower dimensional decoupling with $\delta^{b'}$ in place of $\delta$ in Corollary \ref{cor:LocalImprovedDecoupling}:
\small
\begin{equation*} 
\big( \dashint_{B_H(z,\delta^{-b' \ell})} |f_I|^{p_\ell} * \phi_I \big) \leq C_{\ell}'(C,\underline{D},\delta) \int_{z'} \phi_I(z-z') \big( \sum_{J \in \mathcal{P}(I,\delta^{b'})} \| f_J \|^2_{L^{p_{\ell}}(\phi_{B_H}(z',\delta^{-b' \ell})} \big)^{p_{\ell}/2}.
\end{equation*}
\normalsize
Taking the $p_{\ell}$th root we find for \eqref{eq:Fubini}:
\begin{equation}
\label{eq:FinalAuxEstimate}
\begin{split}
\eqref{eq:Fubini}^{1/p_{\ell}} &\leq C_{\ell}'(C,\underline{D},\delta^{b'}) \big( \int_{z,z' \in \R^3} \big( |f_{I'}|^{p_k - p_{\ell}} * \phi_{I'} \big)(z) \\
&\qquad \times \phi_I(z-z') \big( \sum_{J \in \mathcal{I}(I,\delta^{b'})} \| f_J \|^2_{L^{p_{\ell}}(z'+H,\phi_{B_H}(z',\delta^{-b' \ell}))} \big)^{p_{\ell}/2} \big)^{1/p_{\ell}} \\
&\leq C_{\ell}'(C,\underline{D},\delta^{b'}) \big( \sum_{J \in \mathcal{I}(I,\delta^{b'})} \big( \int_{z,z' \in \R^3} \big( |f_{I'}|^{p_k - p_{\ell}} * \phi_{I'} \big)(z) \\
&\qquad \times \phi_I(z-z') \| f_J \|^{p_\ell}_{L^{p_{\ell}}(\phi_{B_H(z',\delta^{-b' \ell})})} \big)^{2/p_{\ell}} \big)^{\frac{1}{2}}.
\end{split}
\end{equation}
The ultimate estimate follows from Minkowski's inequality since $2 \leq p_{\ell}$. The double integral inside the brackets can be written as
\begin{equation}
\label{eq:AuxEstimateIII}
\begin{split}
&\quad \int_{\R^k} \big( |f_{I'}|^{p_k - p_{\ell}} * \phi_{I'} \big) \big( \phi_I * |f_J|^{p_\ell} *_H \phi_{B_H(0,\delta^{-b' \ell})} \big) \\
&= \int_{\R^k} (|f_{I'}|^{p_k - p_{\ell}} * \phi_I *_H \phi_{B_H(0,\delta^{-b' \ell})} ) ( |f_J|^{p_\ell} * \phi_I ) \\
&\lesssim \int_{\R^k} \big( |f_{I'}|^{p_k - p_{\ell}} * \phi_{I'} \big) \big( |f_J|^{p_{\ell}} * \phi_I \big),
\end{split}
\end{equation}
which follows again by $B_H(0,\delta^{-b' \ell}) \subseteq C \mathcal{U}^o_{I',\gamma}$. Using the uncertainty principle and $\mathcal{U}_I^o \subseteq C \mathcal{U}_J^o$, we find
\begin{equation}
\label{eq:AuxEstimateIV}
|f_J|^{p_{\ell}} * \phi_I \lesssim |f_J|^{p_\ell} * \phi_J * \phi_I \lesssim |f_J|^{p_{\ell}} * \phi_J.
\end{equation}
We merge the implicit constants with $C_\ell'(C,\underline{D})$ to complete the proof.
\end{proof}

\subsection{Proof~of~Theorem~\ref{thm:ImprovedDecoupling}}
\label{section:IterationMomentCurve}

\subsubsection{Asymmetric decoupling constant}
In the following we define asymmetric decoupling constants, which effectively allow us to lower the scale by using lower-dimensional decoupling stated in the previous section. 
We consider two intervals $I$, $I'$ of size $|I| = \delta^a$ and $|I'| = \delta^b$, $a,b \in [0,1]$, which are separated at unit distance. Following \cite{GuoLiYungZorinKranich2021}, we define bilinear decoupling constants as smallest constants, which satisfy the following:
\begin{equation*}
\int_{\R^3} (|f_I|^6 * \phi_I) (|f_{I'}|^6 * \phi_{I'}) \leq \mathcal{M}^{12}_{6,a,b}(\delta) \big( \sum_{J \in \mathcal{I}(I,\delta)} \|  f_J \|^2_{L^{12}} \big)^3 \big( \sum_{J' \in \mathcal{I}(I',\delta)} \| f_{J'} \|^2_{L^{12}} \big)^3.
\end{equation*}
Secondly, we define
\begin{equation*}
\int_{\R^3} (|f_I|^2 * \phi_I) (|f_{I'}|^{10} * \phi_{I'}) \leq \mathcal{M}^{12}_{2,a,b}(\delta) \big( \sum_{J \in \mathcal{I}(I,\delta)} \| f_J \|^2_{L^{12}} \big) \big( \sum_{J' \in \mathcal{I}(I',\delta)} \| f_{J'} \|^2_{L^{12}} \big)^{5}.
\end{equation*}
We have the following as consequence of \eqref{eq:LowerDimensionalDecouplingl1} and \eqref{eq:LowerDimensionalDecouplingl2}:
\begin{lemma}[{Lower~dimensional~decoupling}]
\label{lem:LowerDimDecouplingMomentCurve}
Let $a,b \in [0,1]$ such that $0 \leq a \leq 3b$. Then
\begin{equation}
\label{eq:LowerDimensionalDecouplingMomentCurve}
\mathcal{M}_{2,a,b}(\delta) \leq C_2 \mathcal{M}_{2,3b,b}(\delta).
\end{equation}
If $0 \leq a \leq b$, then the following estimate holds for some $c \in \N$:
\begin{equation}
\label{eq:LowerDimensionalDecouplingMomentCurveII}
\mathcal{M}_{6,a,b}(\delta) \leq C_3 (\log(\delta^{-b}))^c \mathcal{M}_{6,b,b}(\delta).
\end{equation}
\end{lemma}


The following is straight-forward from H\"older's inequality and parabolic rescaling (cf. \cite[Lemma~4.1]{GuoLiYungZorinKranich2021}):
\begin{lemma}[H\"older's~inequality~I]
\label{lem:HoelderInequalityI}
Let $a,b \in [0,1]$. Then
\begin{equation}
\label{eq:HoelderInequalityI}
\mathcal{M}_{2,a,b}(\delta) \leq C_4 \mathcal{M}_{6,a,b}(\delta)^{1/3} D(\delta/ \delta^b)^{2/3}.
\end{equation}
\end{lemma}
%
Another application of H\"older's inequality gives the following (see \cite[Lemma~4.2]{GuoLiYungZorinKranich2021}):
\begin{lemma}[H\"older's~inequality~II]
\label{lem:HoelderInequalityII}
Let $a,b \in [0,1]$. Then
\begin{equation}
\label{eq:HoelderInequalityII}
\mathcal{M}_{6,a,b}(\delta) \leq \mathcal{M}_{2,a,b}(\delta)^{1/2} \mathcal{M}_{2,b,a}(\delta)^{1/2}.
\end{equation}
\end{lemma}

\subsubsection{The decoupling iteration}

By Lemma \ref{lem:LowerDimDecouplingMomentCurve}, \ref{lem:HoelderInequalityI}, and \ref{lem:HoelderInequalityII}, we have the following key iteration step:
\begin{lemma}[Iteration~step~for~the~moment~curve]
\label{lem:KeyIterationMomentCurve}
Let $a,b \in [0,1]$ and $0 < a \leq 3b$. We find
\begin{equation*}
\mathcal{M}_{2,a,b}(\delta) \leq C_5 \mathcal{M}_{2,3b,3b}^{1/3}(\delta) \log(\delta^{-3b})^c D(\delta/\delta^b)^{2/3}.
\end{equation*}
\end{lemma}
\begin{proof}
By successive applications of the aforementioned lemmas, we find
\begin{equation*}
\begin{split}
\mathcal{M}_{2,a,b}(\delta) &\leq C_2 \mathcal{M}_{2,3b,b}(\delta) \leq C_2 C_4 \mathcal{M}_{6,3b,b}(\delta)^{1/3} D(\delta/ \delta^b)^{2/3} \\
&\leq C_2 C_3 C_4 \mathcal{M}_{6,3b,3b}^{1/3}(\delta) \log(\delta^{-3b})^c D(\delta/ \delta^b)^{2/3} \\
&\leq \underbrace{C_2 C_3 C_4}_{C_5} \mathcal{M}_{2,3b,3b}^{1/3}(\delta) \log(\delta^{-3b})^c D(\delta/ \delta^b)^{2/3}.
\end{split}
\end{equation*}
\end{proof}
To make the iteration effective, we initially divide the unit size intervals $I$, $I'$ considered in $B(\delta)$ into $\nu^{-1}$ smaller intervals, and then use the previously established iteration.
Let $\nu = \delta^b$. We choose $\nu = \delta^{1/3^N}$ such that in $N$ iterations of Lemma \ref{lem:KeyIterationMomentCurve} we reach the scale $\delta$, where decoupling becomes trivial. We use the estimate
\begin{equation}
\label{eq:TrivialDecouplingII}
\mathcal{M}_{6,0,0}(\delta) \leq C_6 \nu^{-\frac{1}{2}} \mathcal{M}_{6,b,b}(\delta)
\end{equation}
due to the Cauchy-Schwarz inequality:
\begin{equation*}
\begin{split}
&\quad \big( \int \big( \big| \sum_{J \in \mathcal{I}(I,\delta^b)} f_J \big|^6 * \phi_I \big) \big( \big| \sum_{J' \in \mathcal{I}(I',\delta^b)} f_{J'} \big|^6 * \phi_{I'} \big) \big)^{\frac{1}{6}} \\
&\leq \sum_{\substack{J \in \mathcal{I}(I,\delta^b), \\ J' \in \mathcal{I}(I',\delta^b)}} \big( \int \big( |f_J|^6 * \phi_I \big) \big( |f_{J'}|^6 * \phi_{I'} \big) \big)^{\frac{1}{6}} \\
&\leq \mathcal{M}_{6,b,b}^2(\delta) \sum_{\substack{J \in \mathcal{I}(I,\delta^b), \\ J' \in \mathcal{I}(I',\delta^b)}} \big( \sum_{K \in \mathcal{I}(J,\delta)} \| f_K \|_{L^{12}(\R^3)}^2 \big)^{\frac{1}{2}} \big( \sum_{K' \in \mathcal{I}(J',\delta)} \| f_{K'} \|_{L^{12}(\R^3)}^2 \big)^{\frac{1}{2}} \\
&\leq \mathcal{M}_{6,b,b}^2(\delta) C_6^2 \nu^{-1} \big( \sum_{K \in \mathcal{I}(I,\delta)} \| f_K \|_{L^{12}(\R^3)}^2 \big)^{\frac{1}{2}} \big( \sum_{K' \in \mathcal{I}(I',\delta)} \| f_{K'} \|_{L^{12}(\R^3)}^2 \big)^{\frac{1}{2}}.
\end{split}
\end{equation*}

 By the decoupling result of Bourgain--Demeter--Guth \cite{BourgainDemeterGuth2016} we have 
\begin{equation}
\label{eq:DecouplingBDG}
\mathcal{D}_3(\delta) \leq C_\varepsilon \delta^{-\varepsilon},
\end{equation}
which gives:

\begin{lemma}
\label{lem:IterationI}
Let $N \in \N$. Suppose that $\delta \in 2^{\Z}$ and $\delta^{-\frac{1}{3^N}} \in \N$. Then the following estimate holds:
\begin{equation}
\label{eq:ExplicitIterationI}
D(\delta) \leq C_7 \delta^{\frac{\varepsilon}{3^N} \big( 1 + \frac{2N}{3} - \frac{1}{2 \varepsilon} \big)} \log(\delta^{-1})^{3c} C_\varepsilon^{1- \frac{1}{3^N}} \delta^{-\varepsilon}.
\end{equation}
\end{lemma}
\begin{proof}
We find iterating Lemma \ref{lem:KeyIterationMomentCurve} $N$ times:
\begin{equation}
\label{eq:IterationI}
\mathcal{M}_{2,b,b}(\delta) \leq C_5^2 \mathcal{M}_{6,3^N b, 3^N b}^{1/3^N} \log(\delta^{-1})^{2c} \prod_{j=0}^{N-1} D(\delta/\delta^{3^j b})^{(2/3) \cdot 1/3^j}.
\end{equation}
From the bilinear reduction, we have (here we use $\delta \in 2^{\Z}$)
\begin{equation*}
\mathcal{D}_3(\delta) \leq C_1 \log(\delta^{-1}) \mathcal{M}_{6,0,0}(\delta).
\end{equation*}
We reduce the scale in $\mathcal{M}_{6,0,0}(\delta)$ to $\nu$ by \eqref{eq:TrivialDecouplingII} such that
\begin{equation*}
\mathcal{D}_3(\delta) \leq C_1 C_6 \nu^{-\frac{1}{2}} \log(\delta^{-1}) \mathcal{M}_{6,b,b}(\delta).
\end{equation*}
Now we plug in \eqref{eq:IterationI} to find the following recursive estimate for the linear decoupling constant:
\begin{equation*}
\mathcal{D}_3(\delta) \leq \underbrace{C_1 C_5^2 C_6}_{C_7} \delta^{-\frac{b}{2}} \log(\delta^{-1})^{3c} \prod_{j=0}^{N-1} D(\delta/ \nu^{3^j})^{\frac{2}{3} \cdot \frac{1}{3^j}}.
\end{equation*}
By \eqref{eq:DecouplingBDG}, we find
\begin{equation*}
\begin{split}
D(\delta) &\leq C_7 \delta^{-\frac{b}{2}} \log(\delta^{-1})^{3c} \prod_{j=0}^{N-1} (C_\varepsilon \delta^{-\varepsilon ( 1- 3^{j-N})} )^{\frac{2}{3} \cdot \frac{1}{3^j}} \\
&= C_7 \delta^{-\frac{b}{2}} \log(\delta^{-1})^{3c} C_\varepsilon^{1- \frac{1}{3^N}} \delta^{-\varepsilon (1-\frac{1}{3^N})} \delta^{\varepsilon \frac{2N}{3 \cdot 3^N}} \\
&= C_7 \delta^{\frac{\varepsilon}{3^N} ( 1 + \frac{2N}{3} - \frac{1}{2 \varepsilon})} \log(\delta^{-1})^{3c} C_\varepsilon^{1-\frac{1}{3^N}} \delta^{-\varepsilon}.
\end{split}
\end{equation*}
\end{proof}
In the next step, we choose $N=N(\varepsilon)$, which simplifies the above expression for $\delta \in 2^{\Z}$ and $\delta^{-\frac{1}{3^N}} \in \N$.

\begin{lemma}
\label{lem:ChoiceNMomentCurve}
Let $0<\varepsilon<\varepsilon_0=\varepsilon_0(c)$, and $N \in \N$ such that
\begin{equation}
\label{eq:ChoiceNMomentCurve}
1 + \frac{2N}{3} - \frac{1}{2\varepsilon} \in [\frac{2}{3},2].
\end{equation}
For $\delta \in (\delta_n)_{n=n_0}^\infty$ with $\delta_n = 2^{-n 3^{10N}}$, $n_0 = n_0(c)$, we have the following:
\begin{equation*}
\mathcal{D}_3(\delta) \leq C_7 C_\varepsilon^{1-\frac{1}{3^N}} \delta^{-\varepsilon}.
\end{equation*}

\end{lemma}
\begin{proof}
With the assumptions of Lemma \ref{lem:IterationI} satisfied, we find by \eqref{eq:IterationI}
\begin{equation*}
\mathcal{D}_3(\delta) \leq C_7 \delta^{\frac{\varepsilon}{3^N} \big( 1 + \frac{2N}{3} - \frac{1}{2 \varepsilon} \big)} \log(\delta^{-1})^{3c} C_\varepsilon^{1- \frac{1}{3^N}} \delta^{-\varepsilon}.
\end{equation*}
By \eqref{eq:ChoiceNMomentCurve} this simplifies to
\begin{equation*}
\mathcal{D}_3(\delta) \leq C_7 \log (\delta^{-1})^{3c} \delta^{\frac{2 \varepsilon}{3^N \cdot 3}} C_\varepsilon^{1- \frac{1}{3^N}} \delta^{-\varepsilon}.
\end{equation*}
Since $\delta = 2^{-n \cdot 3^{10 N}}$, we show that for $n \geq n_0(c)$  and $0<\varepsilon<\varepsilon_0$
\begin{equation*}
\log(\delta^{-1})^{3c} \delta^{\frac{2 \varepsilon}{3^N \cdot 3}} \leq 1.
\end{equation*}
First we note that
\begin{equation*}
\delta^{\frac{2 \varepsilon}{3 \cdot 3^N}} \leq \delta^{\frac{1}{3^{2N}}} \leq 2^{-n \cdot 3^{8N }}.
\end{equation*}
Here we use $\varepsilon \sim \frac{1}{N}$, and for $0<\varepsilon<\varepsilon_0$, $N$ becomes large enough to argue like in the above estimate. Moreover,
\begin{equation*}
\log(\delta^{-1})^{3c} \leq n^{3c} 3^{30N c} \log(2)^{3c} \leq n^{3c} 3^{30Nc}.
\end{equation*}
First, we see that
\begin{equation*}
3^{30 Nc} \leq 2^{\log (3) 30 Nc} \leq 2^{\frac{n}{2} 3^{8N}}
\end{equation*}
by choosing $0<\varepsilon<\varepsilon_0(c)$ small enough such that $30 N c \log(3) \leq 3^{8N}/2$ (since $N$ becomes large enough for $N \sim 1/\varepsilon$ such that the inequality holds).
Secondly, we can choose $n \geq n_0(c)$ large enough such that
\begin{equation*}
3 \log_2(n) c \leq \frac{n}{2} \Rightarrow 2^{\log_2(n) 3c} \leq 2^{\frac{n 3^{8N}}{2}}.
\end{equation*}

Then we arrive at the claim
\begin{equation*}
\mathcal{D}_3(\delta) \leq C_7 C_\varepsilon^{1-\frac{1}{3^N}} \delta^{-\varepsilon}.
\end{equation*}
\end{proof}
This absorbs the additional $\log(\delta^{-1})^{3c}$-factor, which is absent in Li's proof of improved decoupling for the parabola \cite{Li2021}. We give the concluding arguments from \cite{Li2021} for self-containedness. Next, we use submultiplicativity to extend this estimate to all $\delta \in \N^{-1}$:
\begin{lemma}
\label{lem:IntermediateResultIIMomentCurve}
Let $0<\varepsilon<\varepsilon_0=\varepsilon_0(c)$ and $n_0 = n_0(c)$ such that Lemma \ref{lem:ChoiceNMomentCurve} is valid. Then there is some $a >0$ such that we find for all $\delta \in \N^{-1}$
\begin{equation}
\label{eq:IntermediateResultIIMomentCurve}
\mathcal{D}_3(\delta) \leq C_{8} 2^{n_0 \cdot 3^{\frac{a}{\varepsilon}}} C_\varepsilon^{1-a/\varepsilon} \delta^{-\varepsilon}.
\end{equation}
\end{lemma}
\begin{proof}
Let $N$ be like in \eqref{eq:ChoiceNMomentCurve} and $\delta \in (\delta_n)_{n=n_0}^\infty = (2^{-n \cdot 3^{10N}})_{n=n_0}^{\infty}$. If $\delta \in (\delta_{n_0},1] \in \N^{-1}$, we use the trivial estimate
\begin{equation*}
\mathcal{D}_3(\delta) \leq \delta^{-1/2} \leq 2^{\frac{n_0}{2} \cdot 3^{10N}}.
\end{equation*}
If $\delta \in (\delta_{n+1},\delta_n]$ for $n \geq n_0$, then submultiplicativity and Lemma \ref{lem:ChoiceNMomentCurve} imply
\begin{equation*}
\begin{split}
\mathcal{D}_3(\delta) &\leq \mathcal{D}_3(\delta_{n+1}) \leq \mathcal{D}_3(\delta_n) \mathcal{D}_3(\delta_{n+1}/ 32 \delta_n) \leq (C_7 C_\varepsilon^{1-1/3^N} \delta_n^{-\varepsilon}) (32 (\delta_n / \delta_{n+1}))^{1/2} \\
&= 32^{1/2} C_7 C_0^{1/2} 2^{\frac{1}{2} \cdot 3^{10 N}} C_\varepsilon^{1- 1/3^N} \delta^{-\varepsilon}.
\end{split}
\end{equation*}
Taking the two estimates together gives
\begin{equation*}
D(\delta) \leq C_{8} 2^{n_0 \cdot 3^{10 N}} C_\varepsilon^{1-1/3^N} \delta^{-\varepsilon}.
\end{equation*}
This estimate holds for all $\delta \in \N^{-1}$ with $N$ given by \eqref{eq:ChoiceNMomentCurve}. Now we simplify by monotonicity in $N$. By the choice of $N$, we have $3^N \leq 3^{a/\varepsilon}$ for some $a$ and $\varepsilon < \varepsilon_0(c)$. We obtain
\begin{equation*}
D(\delta) \leq C_{8} 2^{n_0 \cdot 3^{10 a/\varepsilon}} C_\varepsilon^{1-\frac{1}{3^{a/\varepsilon}}} \delta^{-\varepsilon}.
\end{equation*}
\end{proof}

We bootstrap this bound to find the following:
\begin{lemma}
\label{lem:IntermediateResultMomentCurveIII}
There is $\varepsilon_0 = \varepsilon_0(C_{8},c)$ such that for all $0<\varepsilon<\varepsilon_0(c)$ and $\delta \in \N^{-1}$, we have
\begin{equation*}
\mathcal{D}_3(\delta) \leq 2^{3^{100a/\varepsilon}} \delta^{-\varepsilon}.
\end{equation*}
\end{lemma}
\begin{proof}
Let $P(C,\lambda)$ be the statement that $D(\delta) \leq C \delta^{-\lambda}$ for all $\delta \in \N^{-1}$. Lemma \ref{lem:IntermediateResultIIMomentCurve} implies that for $\varepsilon \in (0,\varepsilon_0(c))$ and $n_0 = n_0(c)$:
\begin{equation*}
P(C_\varepsilon,\varepsilon) \Rightarrow P(C_{8} \cdot 2^{n_0 \cdot 3^{10 a/\varepsilon}} C_\varepsilon^{1-1/3^{a/\varepsilon}}, \varepsilon).
\end{equation*}
After $M$ iterations of the above implication, we obtain
\begin{equation*}
P(C_\varepsilon,\varepsilon) \Rightarrow P((C_{8} \cdot 2^{n_0 \cdot 3^{10 a/\varepsilon}} )^{\sum_{j=0}^{M-1} (1-1/3^{a/\varepsilon})^j} C_\varepsilon^{(1-1/3^{a/\varepsilon})^M}, \varepsilon).
\end{equation*}
We can take limits
\begin{equation*}
C_\varepsilon^{(1-1/3^{a/\varepsilon})^M} \rightarrow_{M \to \infty} 1, \quad \sum_{j=0}^{M-1} (1-1/3^{a/\varepsilon})^j \rightarrow_{M \to \infty} 3^{a/\varepsilon}.
\end{equation*}
Hence, letting $M \to \infty$, we obtain
\begin{equation*}
P(C_{8}^{3^{a/\varepsilon}} \cdot 2^{n_0 \cdot 3^{11a/\varepsilon}}, \varepsilon).
\end{equation*}
By choosing $0<\varepsilon<\varepsilon_0(C_{8},n_0(c))$ we find for all $\delta \in \N^{-1}$
\begin{equation*}
D(\delta) \leq C_{8}^{3^{a/\varepsilon}} 2^{n_0 \cdot 3^{11 a/\varepsilon} } \delta^{-\varepsilon} \leq 2^{3^{100 a/\varepsilon}} \delta^{-\varepsilon}.
\end{equation*}
This finishes the proof.
\end{proof}
In the following we fix $\varepsilon_0 = \varepsilon_0(C_8,c)$ and $a$ such that Lemma \ref{lem:IntermediateResultMomentCurveIII} is valid.

\subsubsection{Proof of Theorem \ref{thm:ImprovedDecoupling}}
We can write for $0<\varepsilon<\varepsilon_0$
\begin{equation}
\label{eq:TripleExponentialBoundII}
D(\delta) \leq A^{A^{1/\varepsilon}} \delta^{-\varepsilon}
\end{equation}
for some $A=A(a) \geq e$. It suffices to prove \eqref{eq:DecouplingConstant3d} with exponentials and logarithms based on $A$.

\begin{proof}[Proof~of~Theorem~\ref{thm:ImprovedDecoupling}]
We optimize \eqref{eq:TripleExponentialBoundII} by choosing $\varepsilon=\varepsilon(\delta)$. 
Let 
\begin{equation}
\label{eq:ChoiceEps}
B = \log_A(1/\delta) > 1, \quad \eta = \log_A (B) - \log_A \log_A(B), \quad \varepsilon = 1/\eta.
\end{equation}
 This leads to the first constraint
\begin{equation}
\label{eq:ConstraintI}
\delta < A^{-1}.
\end{equation}
The constraint on $\varepsilon_0$ translates to
\begin{equation*}
\varepsilon = \frac{1}{\eta} \leq \varepsilon_0 \Rightarrow \frac{1}{\varepsilon_0} \leq \log_A (B / \log_A (B)) \leq \log_A( B) = \log_A (\log_A (1/\delta)).
\end{equation*}
This gives the condition on $\delta$:
\begin{equation}
\label{eq:CondDelta}
\delta < (A^{A^{1/\varepsilon_0}})^{-1} = \delta_0.
\end{equation}
It is straight-forward by \eqref{eq:ChoiceEps} that
\begin{equation*}
A^{1/\varepsilon} \leq \varepsilon \log_A(1/\delta).
\end{equation*}

For this reason we obtain
\begin{equation*}
A^{A^{1/\varepsilon}} \delta^{-\varepsilon} \leq \exp_A (2 \varepsilon \log_A(1/\delta)) \leq \exp_A( \frac{4 \log_A(1/\delta)}{\log_A \log_A (1/\delta)}).
\end{equation*}
In the above display we used that
\begin{equation*}
\varepsilon = \frac{1}{\log_A B - \log_A \log_A B} \leq \frac{2}{\log_A B},
\end{equation*}
which is true for $\log_A B \leq B^{1/2}$. This is true because $A \geq e$ and $B \geq 1$.
Finally, we find with $a = 1/\log(A) \leq 1$ 
\begin{equation*}
\begin{split}
\exp_A( \frac{4 \log_A(1/\delta)}{\log_A \log_A (1/\delta)}) &= \exp_A \big( \frac{4 \log(x)}{\log( a \log(1/\delta))} \big) \\
&\leq \exp_A \big( \frac{8 \log(x)}{\log( \log(1/\delta))} \big) = \exp \big( \frac{4 \log(A) \log(x)}{\log( a \log(1/\delta))} \big).
\end{split}
\end{equation*}
In the estimate we used $\log(a \log(1/\delta)) \geq \log ( \log( 1/\delta))/2$, which amounts to $\delta \leq \exp( - \log(A)^2)$. This is true by \eqref{eq:CondDelta}, and the proof is complete.

\end{proof}

\section*{Acknowledgements}

The analysis in Section \ref{section:LogarithmicImprovement} was previously made available as first version of this preprint and funding by the Deutsche Forschungsgemeinschaft (DFG, German Research Foundation) Project-ID 258734477 – SFB 1173 is gratefully acknowledged.
For the more recently carried out work from Sections \ref{section:Preliminaries}-\ref{section:ProofComplexCurves}, the financial support from the Humboldt foundation (Feodor-Lynen fellowship) and partial support by the NSF grant DMS-2054975 is gratefully acknowledged.

\bibliographystyle{plain}

\begin{thebibliography}{10}

\bibitem{Bourgain1993I}
J.~Bourgain.
\newblock Fourier transform restriction phenomena for certain lattice subsets
  and applications to nonlinear evolution equations. {I}. {S}chr\"{o}dinger
  equations.
\newblock {\em Geom. Funct. Anal.}, 3(2):107--156, 1993.

\bibitem{BourgainDemeter2015}
Jean Bourgain and Ciprian Demeter.
\newblock The proof of the {$l^2$} decoupling conjecture.
\newblock {\em Ann. of Math. (2)}, 182(1):351--389, 2015.

\bibitem{BourgainDemeter2017}
Jean Bourgain and Ciprian Demeter.
\newblock Decouplings for curves and hypersurfaces with nonzero {G}aussian
  curvature.
\newblock {\em J. Anal. Math.}, 133:279--311, 2017.

\bibitem{BourgainDemeterGuth2016}
Jean Bourgain, Ciprian Demeter, and Larry Guth.
\newblock Proof of the main conjecture in {V}inogradov's mean value theorem for
  degrees higher than three.
\newblock {\em Ann. of Math. (2)}, 184(2):633--682, 2016.

\bibitem{GuoLiYung2021}
Shaoming Guo, Zane~Kun Li, and Po-Lam Yung.
\newblock A bilinear proof of decoupling for the cubic moment curve.
\newblock {\em Trans. Amer. Math. Soc.}, 374(8):5405--5432, 2021.

\bibitem{GuoLiYung2021Discrete}
Shaoming {Guo}, Zane~Kun {Li}, and Po-Lam {Yung}.
\newblock {Improved discrete restriction for the parabola}.
\newblock {\em Math. Res. Lett.}, 30(5):1375--1409, 2023.

\bibitem{GuoLiYungZorinKranich2021}
Shaoming Guo, Zane~Kun Li, Po-Lam Yung, and Pavel Zorin-Kranich.
\newblock A short proof of {$\ell^2$} decoupling for the moment curve.
\newblock {\em Amer. J. Math.}, 143(6):1983--1998, 2021.

\bibitem{GuthMaldagueOh2024}
Larry {Guth}, Dominique {Maldague}, and Changkeun {Oh}.
\newblock {$l^2$ decoupling theorem for surfaces in $\mathbb{R}^3$}.
\newblock {\em arXiv e-prints}, page arXiv:2403.18431, March 2024.

\bibitem{GuthMaldagueWang2024}
L. Guth, D. Maldague and H. Wang, 
\newblock {Improved decoupling for the parabola}.
\newblock {\em J. Eur. Math. Soc. (JEMS)}, {\bf 26} (2024), no.~3, 875--917.

\bibitem{Li2020}
Zane~Kun Li.
\newblock Effective {$l^2$} decoupling for the parabola.
\newblock {\em Mathematika}, 66(3):681--712, 2020.
\newblock With an appendix by Jean Bourgain and Li.

\bibitem{Li2021}
Zane~Kun Li.
\newblock An {$l^2$} decoupling interpretation of efficient congruencing: the
  parabola.
\newblock {\em Rev. Mat. Iberoam.}, 37(5):1761--1802, 2021.

\bibitem{PramanikSeeger2007}
Malabika Pramanik and Andreas Seeger.
\newblock {{\(L^p\)}} regularity of averages over curves and bounds for
  associated maximal operators.
\newblock {\em Am. J. Math.}, 129(1):61--103, 2007.

\bibitem{Schippa2025}
Robert Schippa.
\newblock Square function estimates for cones over quadratic manifolds.
\newblock {\em arXiv e-prints}, arXiv:2502.13275.

\bibitem{WooleyA}
Trevor~D. Wooley.
\newblock Vinogradov's mean value theorem via efficient congruencing.
\newblock {\em Ann. of Math. (2)}, 175(3):1575--1627, 2012.

\bibitem{WooleyB}
Trevor~D. Wooley.
\newblock The cubic case of the main conjecture in {V}inogradov's mean value
  theorem.
\newblock {\em Adv. Math.}, 294:532--561, 2016.

\bibitem{Wooley2019}
Trevor~D. Wooley.
\newblock Nested efficient congruencing and relatives of {V}inogradov's
	mean value theorem.
\newblock {\em Proc. Lond. Math. Soc. (3)}, 118:942--1016, 2019.

\end{thebibliography}

\end{document}